\documentclass{article}

\usepackage{amssymb}
\usepackage{amsthm}
\usepackage{rotating}
\usepackage{amsmath}

\usepackage{array} 
\usepackage{arydshln}

\def\lcm{\mbox{\rm lcm}}

\def\Z{\mathbb Z}
\def\N{\mathbb N}
\def\R{\mathbb R}
\def\Q{\mathbb Q}
\def\D{\mathbb D}

\newtheorem{theorem}{Theorem}
\newtheorem*{theorem*}{Theorem \ref{gcdProperty999}'}
\newtheorem*{theorem**}{Theorem \ref{gcdProperty999}''}
\newtheorem*{proposition*}{Proposition \ref{CasPartFauxFois}'}

\newtheorem{proposition}{Proposition}

\newtheorem{corollary}{Corollary}

\newtheorem{definition}{Definition}

\begin{document}

\title{The  Field $\Q$ and the Equality $0.999\ldots = 1$ from Combinatorics of Circular Words and History of Practical Arithmetics}
\markright{$0.999\ldots = 1$ from Circular words and History}
\author{Beno\^{\i}t Rittaud and Laurent  Vivier}

\maketitle

\begin{abstract} We reconsider the classical equality $0.999\ldots = 1$ with the tool of {\em circular words}, that is: finite words whose last letter is assumed to be followed by the first one. Such circular words are naturally embedded with algebraic structures that enlight this problematic equality, allowing it to be considered in $\Q$ rather than in $\R$. We comment early history of such structures, that involves English teachers and accountants of the first part of the {\sc xviii}\textsuperscript{th} century, who appear to be the firsts to assert the equality $0.999\ldots=1$. Their level of understanding show links with Dubinsky {\em et al.}'s {\sc apos} theory in mathematics education. Eventually, we rebuilt the field $\Q$ from circular words, and provide an original proof of the fact that an algebraic integer is either an integer or an irrational number.
\end{abstract}

This article is interested in {\it circular words} (Rittaud \& Vivier, 2012b) as a tool to investigate the famous equality $0.999\ldots=1$ and, more generally, decimal expansion of rational numbers. We propose some perspectives on this famous equality integrating mathematics, history of matematics and mathematics education.

Informally speaking, a circular word  is a finite word whose last letter is assumed to be followed by the first one. (More rigorously, it is a finite word indexed by $\Z/\ell\Z$ instead of $\{1,\ldots,\ell\}$.) Such an object allows us to  remain in $\Q$ instead of $\R$ when investigating the equality $0.999\ldots=1$, hence enables to make use of purely algebraic and combinatorial structures, avoiding analysis and topology.

Section \ref{GenFacts} recalls the relevant facts about usual numeration system in base $b$, where $b\geqslant 2$ is some integer.  For convenience, most examples in the present article are provided in base ten system, that is: with the alphabet $\{0,1,\ldots, 9\}$. (Others examples are given in (Vivier, 2015).)

Sections \ref{MathEduc} and \ref{Histoire} introduce two major extra-mathematical aspects of the equality $0.999\ldots=1$, namely mathematics education and history of mathematics and calculation. A lot of investigations has been made im mathematical education since the 70's and the seminal Tall's work. We focus here on {\sc apos} theory that proposes different mental structures to understand the phenomenon. Also, the contribution made by {\sc xviii}\textsuperscript{th} century English teachers and accountants as regards periodic decimal expansion of rational numbers is mathematically significant. In particular, the two first mentions of the equality $0.999\ldots=1$ appear to come from two of these authors, namely George Brown then Samuel Cunn, with two very different viewpoints.

Section \ref{Combi} is interested in the deeper mathematical structure, setting up modern tools to describe the objects and algorithms involved, namely circular words. It begins with some definition and properties of these (in particular a combinatorial proof of Fermat's Little Theorem), then defines two sets, ${\cal Q}_{\text{\rm WCP}}$ and ${\cal Q}_{\text{\rm DC}}$, both isomorphic to $\Q$ as fields but  set up in a combinatorial way which is free of any reference to the standard construction by pairs of integers. These two sets are quite similar. The first one is best suited for educational purpose, its elements are triples made of a finite word (correspondic to the aperiodic part of the decimal expansion of a rational number), a circular one (for the periodic part), and an integer (for the positioning of the comma). The second one, more theoretical, is best suited for proofs. Its elements are pairs made of a decimal word (i.e. a finite word with a comma, for the aperiodic part)  and a circular word (for the periodic part). Since ${\cal Q}_{\text{\rm WCP}}$ and ${\cal Q}_{\text{\rm DC}}$ are ultimately proved to be fields isomorphic to $\Q$ (Section \ref{Mult}), their elements are to be regarded as combinatorial representations of rational numbers.

Section \ref{Add} deals with the additive structure on ${\cal Q}_{\text{\rm WCP}}$ and ${\cal Q}_{\text{\rm DC}}$, which makes them both isomorphic to $\Q$ as additive groups. Such an additive structure enables to free the equality $0.999\ldots =1$ from any analysis consideration linked to the topology of the real line. (All of this could be extended to the $b$-adic case with only small changes, as it is briefly mention in Section \ref{LeCasbadique}.)

Section \ref{Mult} presents the field structure of ${\cal Q}_{\text{\rm WCP}}$ and ${\cal Q}_{\text{\rm DC}}$ (multiplication and division), making them isomorphic to $\Q$ as fields. It details the limitations of such representations of rational numbers for practical purposes, then investigates some theoretical aspects, culminating in an original proof of the irrationality of numbers like $\sqrt{2}$,  $\sqrt{2}+\sqrt{3}$ and more generally all non-integer roots of a unitary polynomial.

\section{General facts about $b$-adic expansion}\label{GenFacts}

Let us recall the following fundamental result:

\begin{theorem}\label{Main} Let $b \; \in\N$ with $b>1$. A real number has a ultimately periodic $b$-expansion iff it is a rational number.
\end{theorem}

Assuming here the existence of a $b$-adic expansion for any real number, this standard theorem provides a remarkably simple proof of the existence of irrational numbers (consider any non-periodic sequence of digits), also suggesting that the irrational numbers are ``more numerous'' than rational ones. Somehow, Theorem \ref{Main} can also be interpreted an ``unexpected success'' for the $b$-expansion numeration systems: the latter, made up for practical arithmetics, ``unexpectedly'' provides a way for the abstract question of identifying rational numbers among real numbers. (Unfortunately, this ``success'' does not extend to non-rational numbers, since almost nothing is known about the $b$-expansion of numbers like $\sqrt{2}$ or $\pi$.)

The fact that the $b$-expansion of $u/v$ (for $u$ and $v$ integers) is periodic is convincingly proved by the application of the classical algorithm of long division: since, at each step, the remainder is an integer between $0$ and $v-1$, the pigeonhole principle implies that, after all the digits of $u$ has been considered and only $0$s are to be added in the remainder in the next steps, the same remainder will eventually appear twice, hence entering in a periodic loop (possibly reduced to an infinite sequence of quotients equal to $0$).

This proof also shows that the length of the periodic part of the $b$-expansion of $u/v$ is upper bounded by $v-1$ (the number of possible remainders, excluding $0$ and assuming $v>1$). Moreover, it provides an application of the pigeonhole principle, whose first application is generally attributed to Dirichlet, in the end of the {\sc xix}\textsuperscript{th} century, whereas the proof of Theorem \ref{Main} appears in Wallis' {\it Treatise of Algebra} of 1685 (Wallis, 1685, chapter LXXXIX), in which Wallis makes the following precision:

\begin{quote}I have insisted the more particularly on this, because I do not remember that I have found it so considered by any other.
\end{quote}

As regard the pigeonhole principle, Wallis does not state anything special about it, but the argument was already known and used at his time, its most ancient known appearence going back at least to Jean Leurechon in 1622 (Leurechon, 1622; Rittaud \& Heeffer, 2014).

To prove the converse of Theorem \ref{Main}, a simple calculation shows that a number like $0.873873873\ldots$, also written $0.\overline{873}$ in the sequel, is equal to $873/999$, and generalization to all possible type of ultimately periodic $b$-expansion is trivial.

\begin{proposition}\label{UnSurv} Let $v\geqslant 1$ be an integer and $u<v$ prime with $v$. The fraction $u/v$ admits a \underline{purely} periodic $b$-expansion iff $v$ is prime with $b$.
\end{proposition}

By a purely periodic $b$ expansion is meant an expansion of the form $0.\overline{M}$, where $M$ is a finite sequence of digits (which may be regarded as an integer written in base $b$).

Note that, in the case $b$ is equal to ten, the converse of this proposition was proved possibly for the first time by Alexander Malcolm (1730, p. 477) (see Section \ref{AMalcolm}).

\begin{proof} We start by the case $u=1$. Assume first that $1/v$ has a purely periodic $b$-expansion, so $1/v=0.\overline{M}$, with $M$ can be identified with a positive integer. Denoting by $\ell$ the length of $M$, we have
\[\frac{1}{v}=0.\overline{M}=\sum_{n=1}^{+\infty}\frac{M}{b^{n\ell}}=\frac{M}{b^\ell-1},\]
\noindent  so  $b^{\ell-1}b-Mv=1$, so $v$ and $b$ are mutually primes (by B\'ezout's identity).

Now, assume $v$ prime with $b$. The long division algorithm applied to $1$ and $v$ provides the successive equalities
\begin{eqnarray*}
1&=&0v+1\\
1 b&=&q_1v+r_1\\
r_1b&=&q_2v+r_2\\
& \vdots&\\
r_mb&=&q_{m+1}v+r_{m+1}\\
&\vdots&
\end{eqnarray*}
\noindent with $0\leqslant q_m<b$ and $0\leqslant r_m<v$ for any $m$. By the pigeonhole principle, we can find $m\neq m'$ such that $r_{m+1}=r_{m'+1}$. The equalities $r_{m}b=q_{m+1}v+r_{m+1}$ and $r_{m'}b=q_{m'+1}v+r_{m'+1}$ thus imply $(r_{m'}-r_m)b=(q_{m'+1}-q_{m+1})v$. Since $|r_{m'}-r_m|<v$ and  $\lcm(b,v)=bv$, we must have $r_{m'}=r_m$. Hence, by induction, the $b$-expansion of $1/v$ given by the long division algorithm is purely periodic.

Consider now the general case of $u<v$ with $u$ and $v$ mutually primes. If $v$ is prime with $b$, then we already know that $1/v$ is purely periodic, so we easily get that $u/v$ is purely periodic as well. Conversely, assume that $u/v$ is purely periodic. Since $u$ and $v$ have no common divisor, we can find an integer $k$ such that $ku/v$ is of the form $n+1/v$, so $1/v$ is purely periodic, so $v$ and $b$ are mutually primes.\end{proof}

\begin{corollary}\label{Yaunell} For any integer $v$ prime with $b$, there exists $\ell>0$ such that $v$ divides $b^\ell-1$.
\end{corollary}

In standard decimal numeration system, this means that for any integer $v\notin 2\Z\cup 5\Z$, the set $\{9,99,999,9999,\ldots\}$ contains an element which is divided by $v$. For the case of $b$ equal to ten, this quite unexpected fact was also proved by Malcolm (1730, p. 476), in the following way.

\begin{proof}[Proof 1.] Consider the $b$-expansion of $1/v$, which is purely periodic by Proposition \ref{UnSurv}. Writing it as $0.\overline{M}$ with $\ell$ for the length of the period, we therefore have $1/v=M/(b^\ell-1)$, so $Mv=b^\ell-1$ and we are done.\end{proof}

An alternative presentation makes use of the relation $0.\overline{9}=1$ in its general form in base $b$: $0.\overline{\beta}=1$ where $\beta=b-1$.

\begin{proof}[Proof 2.] Apply the long division algorithm to get the $b$-expansion of $1/v$, written as $0.\overline{\beta}/v$, to get successively:
\begin{eqnarray*}
0&=&0v+0\\
0b+\beta&=&q_1v+r_1\\
r_1b+\beta&=&q_2v+r_2\\
&\vdots&\\
r_mb+\beta&=&q_{m+1}v+r_{m+1}\\
&\vdots&
\end{eqnarray*}
\noindent with $0\leqslant q_m<b$ and $0\leqslant r_m<v$ for all $m$. By Proposition \ref{UnSurv}, the $b$-expansion obtained is purely periodic. Hence, writing $\ell$ for the length of the period, we have $r_{\ell+1}=0$ (the rest of the initial division $0=0v+0$), so $0.\beta\cdots\beta=v\times 0.q_1\cdots q_\ell$ (with $\ell$ times the digit $\beta$ in the left side). Multiplying by $b^\ell$ then gives that ${\displaystyle b^\ell-1=v\times\sum_{i=1}^\ell 10^{\ell-i}q_i}$.\end{proof}

Another important property of the sequence $(b^n-1)_n$ is the following one.

\begin{theorem}\label{gcdProperty999} Let $\ell$ and $\ell'$ be two positive integers. The smallest positive integer $n$ such that $b^n-1$ is divided by both $b^\ell-1$ and $b^{\ell'}-1$  is $n=\lcm(\ell,\ell')$.
\end{theorem}

An equivalent form is:

\begin{theorem*} For any positive integers $n$ and $\ell$, $b^n-1$ is divided by $b^\ell-1$ iff $n$ is divided by $\ell$.
\end{theorem*}

\begin{proof} Write $n=k\ell+r$ with $0\leqslant r<\ell$ for the Euclidean division of $n$ by $\ell$. The formula for the sum of the first terms of a geometric sequence gives 
\[\sum_{i=0}^{k-1}b^{i\ell+r}=b^r\frac{b^{k\ell}-1}{b^\ell-1}=\frac{b^n-b^r}{b^\ell-1}=\frac{b^n-1}{b^\ell-1}-\frac{b^r-1}{b^\ell-1},\]
\noindent so $b^\ell-1$ divides $b^n-1$ iff it also divides $b^r-1$ (since the left side is an integer). Since $r<\ell$, this is the case iff $r= 0$.\end{proof}

The previous theorem is coined by Wallis (1685) in the following form:

\begin{theorem**} Let $v$ and $v'$ be denominators of two irreducible fractions, the former (resp. the latter) corresponding to a periodic expansion of length $\ell$ (resp. of length $\ell'$). If $v$ and $v'$ are mutually primes, then the product of the two fractions has a decimal expansion whose periodic part is of length $\lcm(\ell,\ell')$.\end{theorem**}

As regards the product of two rational numbers, the following Proposition, whose proof is a fancy application of standard divisibility criteria, shows that the multiplication of rational numbers is much more difficult to tackle when considering only $b$-expansions. Quite unexpected (especially considering the relative smallness of the length of the product of two fractions as given by Theorem \ref{gcdProperty999}''), it is probably sufficient in itself to explain why practitioners, after some attempts in the first part of the {\sc xviii}\textsuperscript{th} century (see Section \ref{Histoire}), eventually gave up the idea and went back to fractions and decimal approximations for their calculations. (See Theorem \ref{99FauxFois99} for a more general result in the framework of circular words.)

\begin{proposition}\label{CasPartFauxFois} In the decimal numeration system, let $0.\overline{M}$ and $0.\overline{N}$ be two rational numbers, with $M$ and $N$ of length $2$ (that is: $M$ and $N$ are written with two different digits). Let $P$ be the shortest sequence of digits such that $0.\overline{M}\times 0.\overline{N}=0.\overline{P}$. In general, the length of $P$ is equal to $198$.
\end{proposition}

\begin{proof} By Proposition \ref{UnSurv},  $0.\overline{M}$ (resp. $0.\overline{N}$) is equal to some fraction $u/v$ (resp. $u'/v'$) with $v$ (resp. $v'$) prime with $b=10$, so the product is equal to $(uu')/(vv')$. Since $vv'$ is prime with $10$, Proposition \ref{UnSurv} gives that this product can indeed be written $0.\overline{P}$ for some $P$.

The fact that $0.\overline{P}$ can also be written $0.\overline{PP}$ (hence doubling the length) justifies the reference to the minimal possible length for $P$. Write $\ell$ for it. Now, by the proof of the converse of Theorem \ref{Main}, we know that $0.\overline{M}=M/99$, that $0.\overline{N}=N/99$ and that $0.\overline{P}=P/(10^\ell-1)$. Hence, we have $MN/(99\cdot 99)=P/(10^\ell-1)$, so $(99\cdot 99)P=(10^\ell-1)MN$. Thus, whenever $MN$ is prime with $99$ (hence the ``in general'' in the statement of the Proposition) we have that $99\cdot 99$ divides $10^\ell-1$. In this case, the value of $\ell$ is  the smallest positive integer for which $10^\ell-1$ is divided by $99\cdot 99=9^2\times 11^2$.

First, $10^\ell-1$ is divided by $9$ whatever $\ell>0$ is. The quotient is equal to the rep-unit $r_\ell=11\ldots 11$ (with $\ell$ copies of the digit $1$). Since an integer belongs to $9\Z$ iff the sum of its digits belongs to $9\Z$, the rep-unit $r_\ell$ is in $9\Z$ iff $\ell\in 9\Z$.

It remains to show for which values of $\ell$ the rep-unit $r_\ell$ is divided by $11$ twice. An integer belongs to $11\Z$ iff its alternate sum also belongs to $11\Z$, so $r_\ell\in 11\Z$ iff $\ell$ is even. In this case, we have that $r_\ell/11$ is of the form $101010\ldots 101$ (with $\ell-1$ digits). By the same criterion, this latter integer be divisible by $11$ iff the number of $1$s in it belongs to $11\Z$, that is: $\ell/2\in 11\Z$.

Taken together, the preceding conditions show that $10^\ell-1$ is divisible by $99\cdot 99$ iff $\ell$ is a multiple of $9$, $2$ and $11$, so $\ell=198$.\end{proof}

Even the simple example of $(0.\overline{01})^2$ in base ten provides an example of the previous Proposition. Its periodic part is made of the concatenation of words of length $2$ in increasing order from $01$ to $97$, eventually followed by $99$ instead of $98$ (i.e.: the periodic part is $010203\cdots 95969799$).

Since the divisibility criteria by $9$ and $11$ in base ten easily extend to divibility criteria by $b-1$ and $b+1$ in base $b$, the previous proof can be generalized straightforwardly in any base, hence providing the following more general result.

\begin{proposition}\label{99EnBaseb} Let $0.\overline{M}$ and $0.\overline{N}$ be two rational numbers written in base $b$, with $M$ and $N$ of length $2$ (that is: $M$ and $N$ are written with two different digits). Let $P$ be such that $0.\overline{M}\times 0.\overline{N}=0.\overline{P}$. Most of the time, the minimal possible length for $P$ is $2(b^2-1)$.
\end{proposition}

\section{0.999\ldots\ \!=\ 1 in mathematics education}\label{MathEduc}

Studies in mathematics education on the comparison between 0.999... (or $0.\overline{9}$) and 1 are numerous and old (e.g. Tall and Schwarzenberger 1978, Tall 1980, Sierpinska 1985). It is frequently regarded as a key point in the understanding of the set of real numbers, related to several and different notions: completeness (there is no ``hole'' between $1$ and $0.999\ldots$), the notion of limit, infinitesimals (between standard and non-standard analysis), the double representation of finite decimals, the impact on Euclidean geometry (with abscissa on a straight line), potential and actual infinity, etc. The importance of it also arises in high-level mathematics as for example when dealing with Cantor's diagonal argument on the non-denumerability of $\R$.

The equality $0.999\ldots=1$ is a part of mathematics but also of daily life, hence useful to understand not only in mathematical classrooms. Some times ago in Geneva, one of the authors paid with a $50$ euros banknote a good whose price was $50$ swiss francs. The shop assistant proposed an exchange rate of $1.20$ swiss franc for $1$ euro, and started a calculation to determine the amount of swiss francs he had to give back. He calculated  in euros by computing $50-(50/1.2)$, then multiplied the result by $1.2$. His calculator gave the result $9.99999$ and he proposed to give back 10 swiss francs. It is not sure whether the shop assistant understood that, when proposing the deal, he was not, in any way, rounding the result at the advantage of his client.

It should not be thought that this phenomenon comes from the rudimentary calculator used. Let's take the case of a spreadsheet to make the previous calculation (Table \ref{Roundings}). For a cell format Standard or Number with at most 13 decimals, the result is the expected one, but this is not the case with more decimals. What do we do with these writings? Were the previous results rounded? Is the exact value the one obtained with 14 decimals? In short, how to interpret the signs displayed by the computer?

\begin{table}[h!]
\centering
\begin{tabular}{c|l|c}
Cell format & (50–50/1.2)*1.2 display& Some details\\
\hline
Standard & 10&\\
Number (2 decimals)& 10.00 &\\
Number (8 decimals)& 10.00000000& 8 0s after the decimal point\\
Number (13 decimals)& 10.0000000000000&13 0s after the decimal point\\
\hdashline
Number (14 decimals)& 9.99999999999999&14 9s after the decimal point\\
Number (15 decimals)& 9.999999999999990&14 9s after the decimal point\\
Number (16 decimals)&9.9999999999999900&14 9s after the decimal point
\end{tabular}
\caption{Spreadsheet roundings for the calculation $(50-50/1.2)\times 1.2$.}
\label{Roundings}
\end{table}

All studies in mathematics education agree on the fact that it is very difficult to make the students understand the necessity of the equality $0.999\ldots=1$. According to  Weller et al. (2009), preliminary training on periodic decimal expansions contributes to the understanding and control of this equality. The difficulties that arise are about logic, construction of numbers, psychological obstacle  of the strong semiotic difference between the two sides, and conceptual complexity. Also, to preserve a distinction between $0.999\ldots$ and $1$, many scholars as well as students endorse a non-standard analysis viewpoint, writing things like $1=0.\overline{9}+0.\overline{0}1$ (see Vivier, 2011). This shows that the equality $0.999\ldots=1$ heavily relies on the algebraic structure one wishes to define (Rittaud \& Vivier, 2014).

\subsection{{\sc apos} theory}
A particularly efficient way to describe the complexity of the conceptualization required to understand $0.999\ldots=1$ is {\sc apos} theory (Arnon et al., 2014; Dubinsky et al., 2005; Weller et al., 2004). This theory emphasizes on the difficulty for a learner to go from the stage of Action (a finite number of $9$s) to the stage of Process (the digits $9$ continue forever), then to the stage of Object ($0.999\ldots$ becomes $0.\overline{9}$, a ``static'' number), on which we can operate and justify its value $1$. The biggest difficulty seems to be the transition from the stage of Process to the stage of Object. 

Interestingly enough, the different stages of the {\sc apos} theory appears in the historical development of the mathematical tools underlying $0.999\ldots=1$. Section \ref{Histoire}, in which some detailed aspects of the work by {\sc xviii}\textsuperscript{th} century accountants is presented, can therefore be regarded as showing a {\em scheme} in the {\sc apos} sense, in which periodical expansion and periodic parts (independently of the base of numeration) are considered as Objects.\\

The following quote is a good summary of the way the {\sc apos} theory understands the equality $0.\overline{9}=1$ (Dubinsky et al., 2005, pp. 261-262):

\begin{quote} An individual who is limited to a process conception of .999… may see correctly that 1 is not directly produced by the process, but without having encapsulated the process, a conception of the ``value'' of the infinite decimal is meaningless. However, if an individual can see the process as a totality, and then perform an action of evaluation on the sequence .9, .99, .999, \ldots, then it is possible to grasp the fact that the encapsulation of the process is the trancendent object. It is equal to 1 because, once .999\ldots is considered as an object, it is a matter of comparing two static objects, 1 and the object that comes from the encapsulation. It is then reasonable to think of the latter as a number so one can note that the two fixed numbers differ in absolute value by an amount less than any positive number, so this difference can only be zero.\end{quote}

Of course, this is linked with the distinction between actual and potential infinity, as Dubinsky et al. explain in their paper.

More recently, Arnon et al. (2014) suggested, specifically for $0.999\ldots=1$, the introduction of an intermediate stage between Processus and Object: {\it Totality} in which all the $9$s make a single entity. The idea is that $0.\overline{9}$ is regarded as a whole, before having access to the object itself (the number), regardless of the understanding that this number is equal to $1$. Vivier (2011) suggested a quite similar idea by making the distinction between two objects: the number and the period. Indeed, Totality can be seen as the encapsulation of the repetition process of the $9$s to produce the object we  denote by $\overline{9}$ --- in other words, the transition from potential to actual infinity. Afterwards, it remains to establish $0.\overline{9}$ as an object (a number), regardless of the mathematical details.

\subsection{Classical ways to prove the equality}
Teachers are frequently uneasy when confronted to the equality $0.999\ldots=1$. They commonly rely on calculations to justify it to their students, without these calculations being defined beforehand. The elementary combinatorial construction given in Section \ref{Combi} provides a way to overcome this difficulty, thus possibly offering an interesting tool for teachers, even if its efficiency remains to be checked.

Among the classical elementary justifications for $0.999\ldots =1$ inventoried by Tall \& Schwarzenberger (1978) we find the following ones:

\begin{description}

\item[Method 1:] We have $1/3=0.\overline{3}$, so $3\times (1/3)=3\times 0.\overline{3}$, hence $1=0.\overline{9}$.

\item[Method 2:] Write $10\times 0.\overline{9}=9+0.\overline{9}$ to get $9\times 0.\overline{9}=9$, hence $0.\overline{9}=1$.

\item[Method 3:] Let $a=0.\overline{9}$. Dividing $1+a$ by $2$ by the usual long division gives $(1+a)/2=a$, so $a=1$.

\end{description}

Most of the time, when presented to the classroom, none of these calculations are properly defined in the first place. All of them contain a lot of implicit assumptions. In method 1, it is assumed that $0.\overline{3}$ (and its triple) represents a rational number. In methods 2 and 3, it is assumed that the equalities between infinite expansions can be simplified under the rule $a+b=a+c\Longrightarrow b=c$.   According to Tall and Schwarzenberger (1978), method 3 is the most legitimate, since it is the only one in which calculations are rightfully made from the left to the right. In this article, it is also suggested the following alternative explanation for $0.\overline{9}=1$: we have $1/9=0.\overline{1}$, $2/9=0.\overline{2}$ and so on until $8/9=0.\overline{8}$, hence $9/9=0.\overline{9}$. This could be justified by the form of the long division derived from the alternative Euclidean division $a=qb+r$ in which $0<r\leqslant b$ instead of $0\leqslant r<b$. Such an alternative long division always provides the quotient in a decimal form that never ends.

An alternative proof, for which we did not find any reference, comes from geometry: on the real line (or even on the {\em rational} line), the segment $I=[0.\overline{9},1]$ has no point in its interior (since there is no decimal expansion between $0.\overline{9}$ and $1$), hence $I$ reduces to a single point. Quite convincing in itself, such an argument still needs real analysis to be properly completed.

Other procedures can be set up, involving topology and analysis, like summation of series (Njomgang Ngansop \& Durand-Guerrier, 2014 ; Tall et Vinner, 1981) or the use of the separation axiom for the standard topology of the real line $\forall\ \varepsilon>0,\ |a-b|<\varepsilon\Longrightarrow a=b$ made in (Dubinsky et al. 2005). It is this latter property that lies behind Zeno's paradox (see Fishbein 2001). More generally, Wilhelmi et al. (2007) present several ways to justify that two numbers are equal. Nevertheless, they rely explicitely on the construction of the field $\R$ and its general properties.

In fact, contrarily to a quite common belief, the equality $0.999\ldots=1$ is not necessarily linked to the structure of the real line $\R$, and can be regarded as a fundamental property of the field $\Q$ alone. Therefore, staying in $\Q$ not only focus on an essential issue, but also avoids technical considerations about analysis or topological properties of the real line.

\subsection{Students' difficulties}
Therefore, all these methods (for which maybe we should speak of {\it evidences} rather than {\it proofs}) rely on some properties of a structure already set up, in general $\R$. Moreover, several studies (Mena et al., 2014 ; Njomgang Ngansop \& Durand-Guerrier, 2014 ; Tall \& Vinner, 1981) explain that, in general, these arguments are not convincing for students, even if they frequently acknowledge their validity. The point is that the semiotic opposition between $0.\overline{9}$ and $1$ appears to be too tough.

Besides, it is highly significant that the proportion of mathematically skilled people for which $0.\overline{9}=1$ remains around 60 \%, a figure quite independent from time, country or specific preparation:

\begin{itemize}

\item At an undergraduate level, 28 students out of 43 (65~\%) assert that $0.\overline{9}<1$ (Vivier 2011) and Tall (1980) reports 20 students out of 36, so a proportion of 56 \%.

\item Mena et al. (2014) find 23 teachers and student teachers out of 40 (57.5~\%) in favor of the inequality. In the same study, the authors find 12 out of 19 teachers (63 \%) enrolled in {\it maestria} of mathematical teaching and for which $0.\overline{9}<1$.

\end{itemize}

For the sake of completeness, let us also mention the study made on a non-mathematician population of 204 students-teachers of primary school (Weller et al., 2009). This study found that 73.5~\% of them believed that the inequality holds, a significantly bigger proportion than for mathematically skilled people, but  by a  rather small margin. Unsurprisingly, the only population for which the margin is really large is the one of scholars of secondary school: 100~\% out of 113 scholars assert that $0.\overline{9}<1$  (Vivier 2011).

An explanation could be the direct opposition of $0.\overline{9}=1$ with the knowledge, firmly established for numerous years, about comparison of decimal expressions. The necessary reset of this knowledge is more difficult to do than it is for other results like $(-1)\times(-1)=1$, which does not oppose any prior knowledge.

To try and overcome this difficulty, we prove that the equality $0.\overline{9}=1$ derives from the need to make use of infinite digit sequences as numbers. Following (Rittaud \& Vivier, 2014), we argue that the equality $0.\overline{9}=1$ consists in a technology (in the sense of Chevallard (1999)), typical of $\Q$ but generally not made explicit.

An experiment made in (Rittaud \& Vivier, 2014) involved 29 undergraduate students in France and relied on ancient knowledge. After a study, somehow too fast, of the usual summation algorithm, it is  observed that if the decimal expansion of a number $a$ has a nontrivial periodic part, then the computation of $0.\overline{9}+a$ provides the same result as $1+a$ (Richmann, 1999), hence the need to assume $0.\overline{9}=1$ to preserve standard algebra. The interest of this is that it is based solely on ancient algebraic knowledge.

From another perspective, the reorganization of the knowledge of the student does not necessarily imply the acceptance of the equality $0.\overline{9}=1$. Indeed, in the context of non-standard analysis it is concievable to write $0.999\ldots<1$, the difference being an infinitesimal commonly written as $0.000\ldots 1$ (with the idea that the expression contains infinitely many $0$s) by students. (see also Margolinas 1988). Nevertheless, the theory is difficult, and despite some promising attempts to introduce it in the curriculum (Artigue 1991, Hodgson 1994), non-standard analysis remains marginal. Still, some searchers are trying to develop it as an enlightening way to look at numbers (see Katz \& Katz, 2010a, 2010b). Such a point of view is important to consider when trying to understand the way students develop personal concepts that can sometimes oppose those of the standard curriculum (Ely, 2010). For example, Manfreda Kolar \& Hodnik Čadež (2012, pp. 404-405) asked 93 primary preservice teachers the question ``what is the largest number?'' and got once the answer $99\ldots$. To the question ``What number is closest to the number $0.5$?'', 67 students\footnote{That is 68\%, which is close to the percentages presented in the previous section.} answered $0.4999\ldots$ and 3 answered $0.500\ldots 1$. These answers are regarded as a way to deal with potential infinity, but we could see them as echoing a non-standard conception of numbers as well.

\subsection{From digit representations to numbers}
The initial question is: how to make numbers from infinite sequences of digits. Such sequences of digits are not sufficient {\it per se}, as Chevallard (1989) explained in his definition of system of numbers, which contains the necessity of being able to compare and make basic operations with usual properties. Hence, the question becomes: how can we operationalize the set of decimal expressions?

In 1971, the French grade 8 curriculum (``classe de quatri\`eme'') made an attempt for this. The point was that the underlying motivation was to set up the field $\R$, hence the attempt led to difficult considerations of approximations. One can reasonably assert that such an approach remained purely theoretical, without any practical use in the classroom. Recently, Fardin \& Li (2021) proposed a more operational definition that allows to multiply two infinite sequences of digits starting from the {\it left}, hence avoiding the problem of approximation that arise when extending the usual algorithm which starts from the right. Their construction consists in a non-trivial extension of the corresponding natural idea for addition starting from the left (in which the carry is handled by looking separately the case in which the sum of the $k$-th digits is ultimately equal to $9$ for all $k$). The point is that there is no focus on a concrete method to identify the appearance of this case.

Considering the field $\Q$ instead of $\R$ allows us to consider only ultimately periodic sequences of digits (see Theorem \ref{Main}). At first, these are semiotic representations in which the signs (digits and decimal points) are numerical, and we wish to make them authentic numbers. Therefore we cannot make use of an approch like Anatriello \& Vincenzi (2019), in which operations are made using the register of fractions.

According to Yopp et al. (2011), teachers for the end grades of primary school should have some knowledge about the equality $0.\overline{9}=1$, since it has an impact on arithmetic understanding of rational numbers. Such a conclusion may be extended to a broader set of people. Indeed, it is observed in (Rittaud \& Vivier, 2014) that no student for which $0.\overline{9}=1$ shares any infinitesimal way of thinking (Margolinas, 1988) like $0.\overline{0}1$. Also, at a primary level, a teacher may have to deal with the tricky situation of $0.\overline{9}$ and $1$ considered by some pupil as a counterexample to the fact that, between two numbers, there is always a third one. 

Duval (1996) showed how important it is for a mathematical object to be understood in two different registers. Fractions and (ultimately) periodic decimal expansion are two numerical registers for rational numbers. They should be articulated to each other, whereas secondary school mostly consider only fractions. Not only this could impede the cognitive appropriation of the object, but it  also paves the way to the confusion between the mathematical object (a rational number) and the unique data structure in use to represent it (a fraction).

Inspired by Weller et al (2009), Voskoglou (2013) made an attempt to link the two different representations. In particular, his experiment aims at identifying both fractionary and decimal representations of rational numbers. At it seems, this work on several representation systems is helpful for a better understanding of the notions of rational and irrrational numbers.
 
In the framework of {\sc apos} theory, the study made by Weller et al (2009) (see also Arnon et al (2014) chapter 8) is a strong case in favour of an operationalization of periodic sequences of digits. It argues for the teaching of rational numbers in both fractional and decimal registers, proposing operations on decimal expressions with the help of a software. This software makes all the computations by the use of fractions, something the user cannot notice since the interface only shows the decimal register. The study shows a clear improving of knowledge, both quantitatively and qualitatively, about rational numbers in decimal expansion, especially as regards the two expansions of decimal numbers (experimental group of 77 individuals, control group of 127 students-professors of primary school). These results can be understood as the institution of periodic sequences of digits as numbers, since it becomes possible to perform basic arithmetic operations on them (see also Yopp et al 2011 and Vivier 2011). The object ``number'' can therefore emerge from this. The link between these numbers and fractions is done as well, an essential step to avoid the constitution of two separates and somehow ``parallel'' sets of numbers instead of only one which can be studied with two registers.

\section{Periodic expansions in the history of practical arithmetics}\label{Histoire}

The equality $0.999\ldots=1$ and its links to circular words theory is part of the history of the more general subject of the present paper, noticeably investigated by Maarten Bullynck (2009): the representation of rational numbers by the way of decimal numeration instead of fractions. The starting point is the result recalled in Theorem \ref{Main}: the decimal expansion (or, should we write, \underline{a} decimal expansion) of any (positive) real number $x$ is ultimately periodic iff $x$ is rational. It seems that it took some time before the importance or such a result is recognized. For example, Simon Stevin, in his famous 1585 text entitled {\em La Disme} (Stevin \& Girard, 1625) about decimal representation of numbers and computation, ignores the periodicity property of rational numbers (see especially Proposition IV, Nota 1). An explanation could be that it took time for mathematicians to become interested in decimal representation, which is more a subject for accountants in the first place, and that accountants themselves were more interested in decimal approximation than in theoretical considerations about numbers. Also, the almost intractable problem of the multiplication (see Proposition \ref{CasPartFauxFois} and, more generally, Theorem \ref{99FauxFois99}) could also have been a huge practical restraint.

As already mentioned in Section \ref{GenFacts}, the beginnning of deep mathematical investigations from the equivalence between rational numbers and ultimately periodic expansion, together with some complements on the size of the period, is probably due to John Wallis in his {\em Treatise of Algebra} of 1685 (Wallis, 1685, chapter LXXXIX). Wallis did not noticed the equality $0.999\ldots=1$, nor he investigated the effects of arithmetic operations on decimal expansion of rational numbers (even if his chapter VIII is quite close to this). His study is primarily about the length of the periodic part of the decimal expansion of a fraction. His short work (Wallis, 1685, p. 326-327) mainly indicates that this length $\ell$ is given by writing the denominator as $2^a\cdot 5^b\cdot v$ with biggest possible integers $a$ and $b$, then looking for the smallest $\ell$ such that $v$ divides the number $99\cdots 99$ (with $\ell$ times the digit $9$ --- see Corollary \ref{Yaunell}). He also mentions that if the decimal expansion of a fraction of denominator $u$ (resp. $u'$) has a periodic part of length $\ell$ (resp. $\ell'$) and that $u$ and $u'$ are mutually primes, then the length of the periodic part of the decimal expansion of a fraction of denominator $uu'$ is $\lcm(\ell,\ell')$. (See our Theorem \ref{gcdProperty999}.) Also, Wallis'  final remark that ``What have been said of Decimal Fractions, may, with very little alteration, be easily accomodated to Sexagesimal Fractions'' show that he perfectly understands that all his results have a base-$b$ counterpart. Also, Wallis states that, for square roots like $\sqrt{2}$, ``we have not the like recurrence of the numeral Figures in the same order'' (i.e. there is no periodic pattern in the sequence of decimals).

The history of the sequels of Theorem \ref{Main} split into two differents parts, theoretical and pratical, the latter one being the most sensible for our purpose. We will present the details of it relevant for the present article, postponing to a forthcoming paper the presentation of the full story of the consequences of Theorem \ref{Main} to what was called ``practical arithmetics'' in the {\sc xviii}\textsuperscript{th} century.  Before going into this, for the record, let us present some aspects of the theoretical aspects investigated after Wallis (see Bullynck, 2009).

Wallis' most famous followers in studying the subject are: Johann Heinrich Lambert, whose first attempt to prove the irrationality of $\pi$, before his more fruitful approach by continued fractions, consists in trying to show that its decimal expansion is not periodic; Leonhard Euler, who presents the general properties of decimal expansions rational numbers with the formalism of series (Euler, 1822)\footnote{see chapter XI, section III, points 523-524, then chapter XII - see especially the scholium in pages 174-175}; Johann Bernoulli and eventually Carl Gauss, who provides in his famous {\em Disquisitiones Arithmeticae} of 1801 the mathematical background to understand the properties of decimal expansions (or, more generally, $b$-expansions, where $b>1$ is any integer) of rational numbers: length of the periodic part, effect of the multiplication by an integer\ldots

Then the story seems to end, since Gauss proves that all that can be said on the subject mainly rely on Fermat's Little Theorem and its consequences. Nowadays, $b$-expansion of rational numbers are more regarded as recreational mathematics (as it is already the case for example in (Rademacher \& Toeplitz, 1930, pp. 113-126)). Nevertheless, alternative numeration systems investigated from the second part of the {\sc xx}\textsuperscript{th} century (with the seminal works of Rényi (1857), Parry (1960) and others) provide new scopes for these old questions (Rittaud \& Vivier, 2012; Rittaud \& Vivier, 2011; Rittaud, to appear).

Now, a more relevant part of the story surrounding Theorem \ref{Main} in our context is about mostly forgotten English writers concerned with teaching and accounting necessities during the first part of the {\sc xviii}\textsuperscript{th} century. Starting from Wallis' initial ideas, these numerous authors develop several new practical algorithms to deal with periodic decimal expansion: addition, multiplication, division, and even elevation to the $n$-th power. Here, to remain in the bounds of our purpose, we limit ourselves to a brief presentation of some of the works of five authors: Brown, Cunn, Hatton and Marsh. (The forthcoming paper on the full story involves many more authors.)

\subsection{George Brown: a mixed numeration system}\label{GBrown}

We may argue that George Brown is the first to face the equality $0.999\ldots=1$, since he is seemingly the first to operate with periodic expansions in his {\em System of decimal Arithmetick} (Brown, 1701). Nevertheless, he is not concerned with such an equality, because of his very clever interpretation of ``Infinites'' (i.e. infinite periodic expansions) that allows him to consider them as finite expressions. In his explanation of it right from the beginning of his study of infinite periodic part (p. 12), the part which is closest to correspond to $0.999\ldots=1$ is:

\begin{quote}you must reckon the figure next the Right hand of an Infinite, as Ninth parts, but not as Tenths of the next preceding Unites~; and for that cause, you must in Addition and Multiplication, carry one for every Nine of the Sum, or Product ; and in Substraction and Division, reckon upon Nine, for every one of the borrowed, or imaginary prefix.\end{quote}

Hence, when Brown writes the decimal expansion of $1/240$ as $0.00416$, this latter expression is to be interpreted as ${\displaystyle \frac{4}{10^3}+\frac{1}{10^4}+\frac{6}{9\cdot 10^4}}$.

Brown representation is an exact one, essentially equivalent to the WCP-representation given in Section \ref{Representation1}. The only issue is the ambiguity of notation that makes difficult to distinguish numbers like $3/10$ and $3/9$ (which could be both written as $0.3$). Most importantly, in Brown's mixed numeration system, in which the last digit is to be considered as ninths parts of the previous one, there is no need to consider anything close to $0.999\ldots=1$.

In some way, it is a little bit disappointing that Brown, the first real user of infinite periodic decimal expansion, is so clever that he overcomes right from the beginning any reference to the puzzling equality. The only calculation in his book in which we could potentially recognize it (even if, so, it is quite abusive to interpret it this way) is the  one of Figure \ref{Brown}, corresponding to the sum $0.001041\overline{6}+0.002083\overline{3}=0.003125$. 

\begin{figure}[h!]
\centering
\includegraphics[height=3cm]{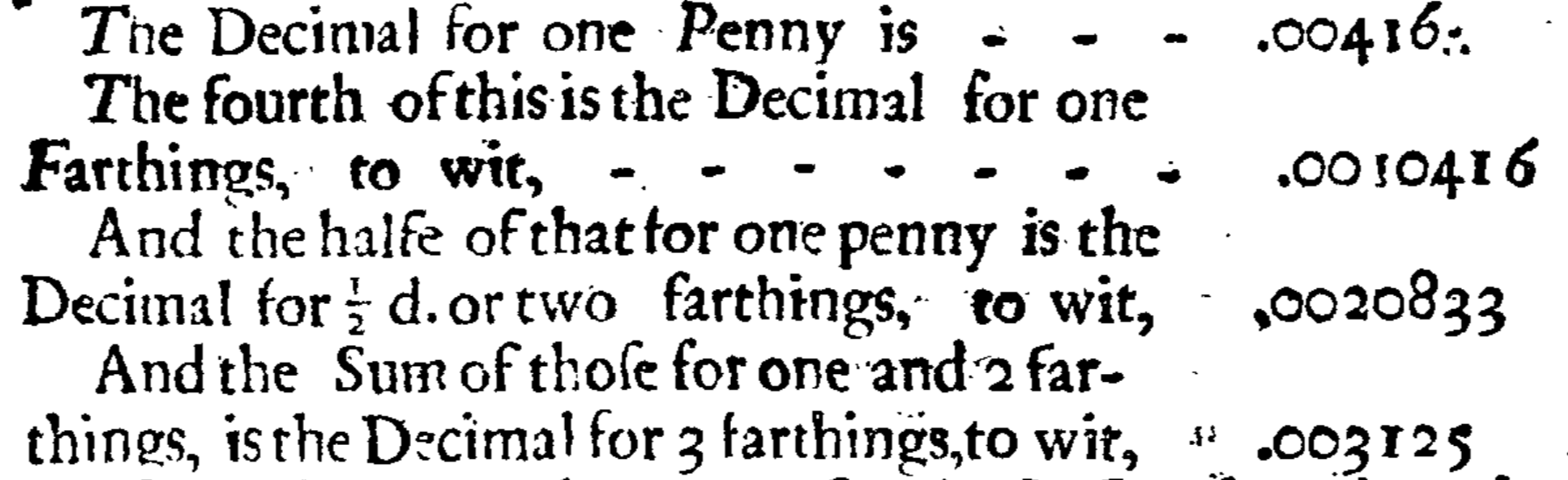}
\caption{Brown's calculation for the sum of two periodic decimal expansions.}
\label{Brown}
\end{figure}

In this figure, ``decimal for one Penny'' stands for the decimal expression of the value of one penny expressed in pound: 1 penny is one twelfth of a shilling and 1 shilling is one twentieth of a pound, so a penny is a $(1/12)\times (1/20)=1/240$-th of a pound, so $0.0041\overline{6}$.

Most of the time, Brown remains in this context of conversion between units of English money of the time, in which appears mainly periodic expansion of length $1$, frequently $\overline{3}$ or $\overline{6}$. He is at ease to make the correspondence between such an expression and a fraction, as in Figure \ref{Brown2} (p. 46) of $29.1\overline{6}$, straightforwardly identified with ${\displaystyle 29+\frac{1}{10}+\frac{2/3}{10}}$. (The written numbers in the left stand for the product $3 5\times 50=1750$ and $1750/6=29,166\ldots$.)

\begin{figure}[h!]
\centering
\includegraphics[height=2.5cm]{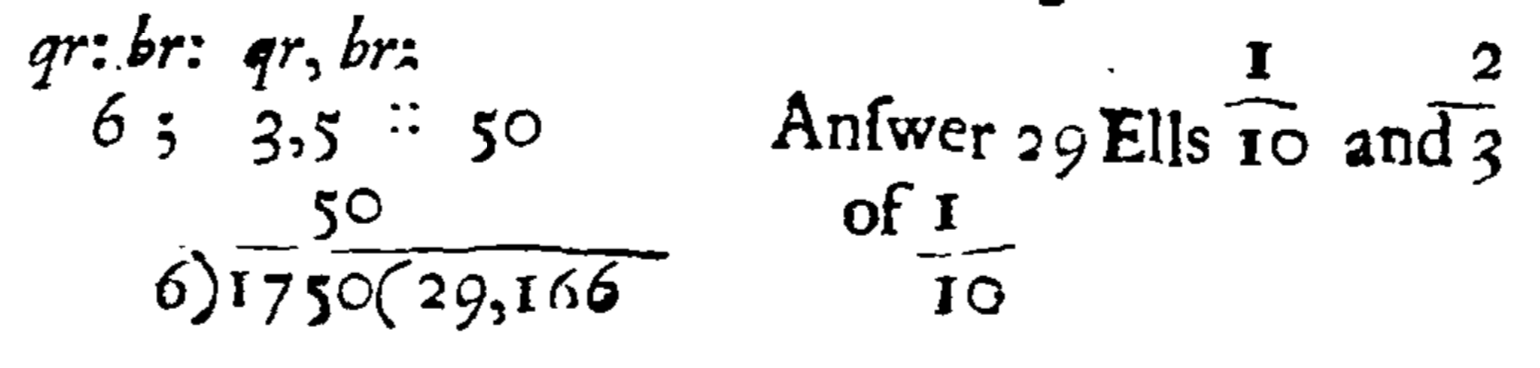}
\caption{Brown's identification of $29.1\overline{6}$ with $29+1/10+(2/3)\times(1/10)$.}
\label{Brown2}
\end{figure}

When possibly confronted with periodic parts of length greater than 1, Brown avoids it by an approximation, without any elaboration. In one case (p. 46), presented in Figure \ref{Brown3}, he uses a kind of improvisation: when computing $297.5/11$, which is $27.0\overline{45}$, Brown writes ${\displaystyle 27+\frac{1}{22}}$ instead, the fractional part being deduced from the rest $0.5$ to be divided by the divisor 11.

\begin{figure}[h!]
\centering
\includegraphics[height=1.5cm]{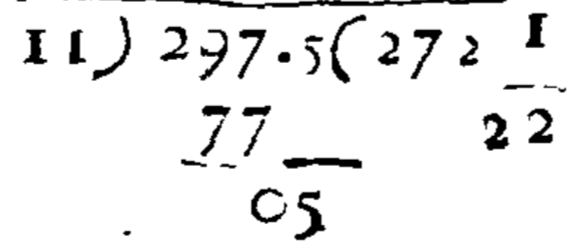}
\caption{Brown's avoidance of periodic expansion of length more than 1.}
\label{Brown3}
\end{figure}

\subsection{Samuel Cunn: operating on periodic expansions and 0.999\ldots}

The work of Samuel Cunn (Cunn, 1714) in the field of periodic decimal expansions is the second substantial one since Wallis. Contrarily to Brown, Cunn mentions the previous work of Wallis, as well as some minor considerations made by two other authors, but seems unaware of Brown's book and has a very different standpoint. Cunn's full chapter on the subject shows a clear understanding of many issues. It is quite frustrating that he does not provide any proof for the many results he states, all of them being very accurate.

In his preliminary definitions, Cunn makes the difference between decimal an non-decimal numbers, what he calls {\it terminates} and {\it interminates}, noticing two pages later that ``Every Terminate may be consider'd as Interminate, by making Cyphers the Repetend''.\footnote{He also unifies finite and periodic decimal expansions by defining the notion of {\it compleat decimal}, but without using it afterwards; he also avoids talking explicitely about aperiodic expansions, opposing compleat decimals to {\it approximate} ones, ``that hath some places true, but all the following ones uncertain''. This separation may be understood as algorithmical, in the sense that approximate decimals are those for which there is no obvious rule for their sequence of digits.} 

The way Cunn understands ultimately periodic expansions is very combinatorial, being in particular highly interested in the length of the periodic part of an expansion. Cunn's perspective is close to what we call the WCP representation of numbers in Section \ref{Representation1}. He mentions explicitely what corresponds to the shift identification and circular powers identification. He understands it both ways, that is: a repetend like 56 can be extended as 565656, but also ``if the Repetend consists of some other Repetend of fewer places, retain the latter only''. He also states in a general forms results given by Wallis, as well as the equality $0.\overline{N}=N/99\ldots 9$ (with as many $9$s as the length of $N$). He explains precisely, and with examples, how to add and substract ultimately periodic expansion, but also how to multiply and divide them.

With such an approach, it is quite inevitable that Cunn is also interested specifically in $0.999\ldots$, writing what is possibly the first explicit remark on this expression, made as a theorem stating that (Cunn, 1714, p. 63):

\begin{quote} Instead of .9999 an so on continually, put an Unit, for that is either equal to this, or else wants of it less than any thing assignable.\end{quote}

Unfortunately, this observation is nowhere followed by anything else, neither for a more rigorous proof nor in subsequent rules and examples. Hence, we cannot know for sure the reason why Cunn made such a fundamental observation. Even some of his subsequent examples which could explicitely require the identification of $0.999\ldots$ and $1$ (like $3.17\overline{6}\times 0.3=0.9530$) are treated with a method that does not need it.

After addition and substraction, Cunn goes for multiplication, which is much more difficult as we already mentioned in Proposition \ref{CasPartFauxFois} (see also Section \ref{Mult}). He states his results in increasing complexity, eventually providing complete algorithms, valid in all possible cases. 
Two statements given by Cunn (1714, p. 66) are particularly striking:
\begin{quote} If any required Root of some terminate Number be not exactly had from the Places given, it cannot be exactly had.

If any required Root of some circulating Expression doth not repeat from the Repetend once used, it cannot repeat at all.
\end{quote}

Even if these sentences are quite imprecise, we can interpret them as stating that if $x\in\D$ then $\sqrt[n]{x}$ is either decimal or irrational, and that the same is true for $x\in\Q$. It is probable that, here, Cunn is simply restating the corresponding remark made by Wallis (see the beginning of section \ref{Histoire}). Unfortunately, he does not try to go beyond this statement and provide any clue for a proof in the spirit of his {\it repetends}. (See Theorem \ref{IrrationaliteUnitaires} for such a proof of a more general statement.)

\subsection{Edward Hatton: accounting and recreational mathematics}

Hatton's presentation of decimal arithmetics can be seen as a rationalization of some ideas on numbers sometimes rather naive. For example, in (Hatton, 1721, p. 131), Hatton defines decimal numbers as fractions whose denominator is a power of $10$, then goes for the decimal expansion of them, and eventually consider decimal numbers as possibly ``infinite'' (i.e. with infinitely many digits). One may interpret it as an implicit shift from the arithmetical definition of decimal numbers (fraction of denominator $10^n$) to a combinatorial one (a sequence of digits), the latter one being praised by Hatton ``because so like to an intire Number'' (Hatton, 1721, p. 13).

In (Hatton, 1721, p. 147), Hatton presents what is possibly the first historical example of an explicit computation leading to a number ending with infinitely many 9s, namely the product $0.1256\overline{4}\times0.00009=0.000011307\overline{9}$, as shown in Figure \ref{Hatton999}. (In Hatton's notation, the ``{\it r} 1'' means that the one last digit is to be repeated {\em ad infinitum}, as he explains it in (Hatton, 1721, p. 134).)

\begin{figure}[h!]
\centering
\includegraphics[height=3cm]{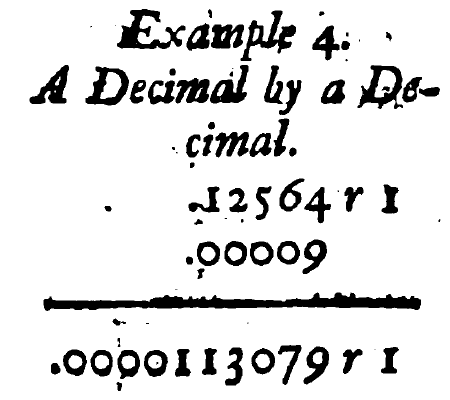}
\caption{Hatton's computation of $0.1256\overline{4}\times0.00009=0.000011307\overline{9}$.}
\label{Hatton999}
\end{figure}

As before, both factors are regarded as ``Decimals''. The justification for this calculation consists in reporting adequately the carry for the periodic part. To quote Hatton, who is quite clear here (Hatton, 1721, p. 147-148):

\begin{quote} {\it Note}, That in the fourth Example, because the 4 is repeated {\it ad infinitum}, therefore I say, 9 times 4 is 36, and 3 (which would be carry'd if you actually put down another 4) is 39~; put down 9, and carry 3. Now if you had put down and multiplied 100 Fours of those repeated, so many Nines would also be repeated in the Product~; but for brevity-sake I only put down one of each with an {\it r}.  \end{quote}

The main difference with Cunn is the fact that Hatton introduces the sign $r1$ to unify the infinite sequence of digits. 

 Hatton is not really interested in a general theory of calculation with periodic expressions. In his second work on the subject (Hatton, 1728), he sees the question merely as practical or recreational. He provides a few examples for which a quite naive algorithm is sufficient (as it is already the case in (Hatton, 1721) even if there are more examples and details in the latter). Still, Hatton's 1721 book proves that Hatton does understand how to calculate with periodic expansions.
 
 Nevertheless, one of his calculations makes it quite clear that identifying $0.\overline{9}$ and $1$ is beyond his scope (see Figure \ref{Hatton189}). 
 
\begin{figure}[h!]
\centering
\includegraphics[height=5cm]{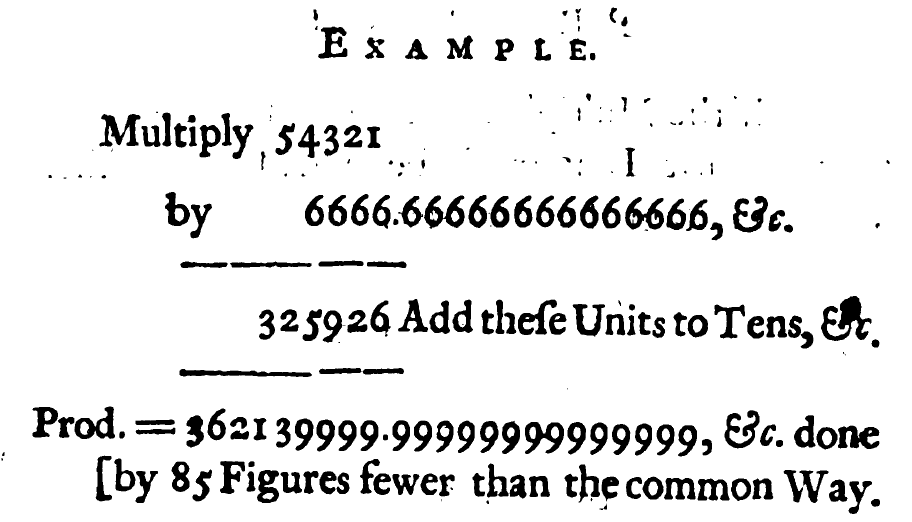}
\caption{Hatton's non-identification of $362139999.\overline{9}$ with $362140000$.}
\label{Hatton189}
\end{figure}

\subsection{Alexander Malcolm: $0.\overline{9}=1$ and beyond}\label{AMalcolm}

In 1730, Alexander Malcolm (1730, p. 472) provides what is possibly the first general statement about the existence of two decimal expansions for decimal numbers (Figure \ref{Malcolm}), as well as a rigorous proof of it.

\begin{figure}[h!]
\centering
\includegraphics[height=2.3cm]{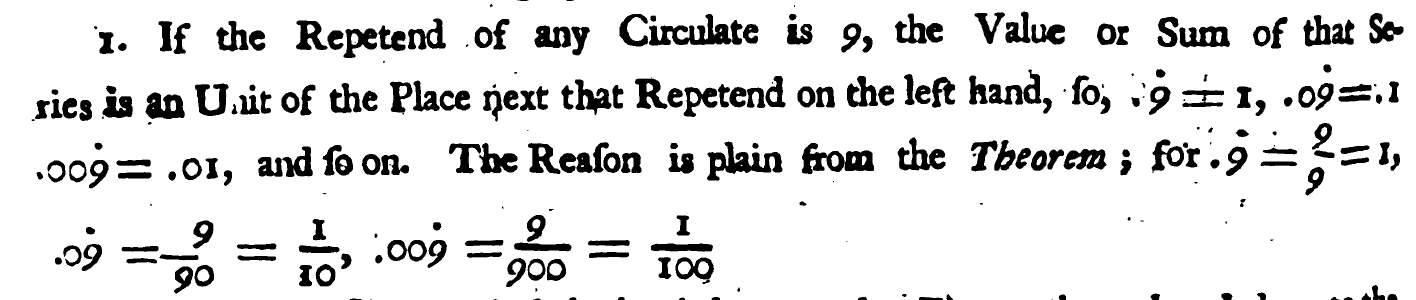}
\caption{Malcolm's assertion on decimal expansions of decimal numbers.}
\label{Malcolm}
\end{figure}

By the ``Repetend of any Circulates'', Malcolm means the periodic part of the decimal expansion of a number. Also, Malcolm writes $\dot{9}$ for our $\overline{9}$, hence the equality $.\dot{9}=1$ corresponds exactly to our $0.\overline{9}=1$. The theorem he refers to states the general correspondence between rational numbers and ultimately periodic expansions, based on the equality $0.\overline{M}=M/(10^{\ell}-1)$, where $\ell$ is the number of digits in $M$. Malcolm proves this theorem at length, by the use of the formula of the sum of the terms of a geometric sequence.

\subsection{John Marsh}\label{JMarsh}

Wherever all previous authors consider periodic decimal expansions rather as an aspect of decimal arithmetics among others, John Marsh is the first, in 1742, to write a book fully devoted to this single  notion (Marsh, 1742). (For an extensive presentation of his work, we refer to (Melville, 2018).)  Marsh is fully aware of the authors before him, but want to get rid of some mistakes he found in their works. He produces general algorithms for multiplication and division, and understands the problems arising by rapidly increasing length of periodic decimal parts in calculations (see our Proposition \ref{CasPartFauxFois}).

Marsh does understand very well the equality $0.999\ldots=1$, and states the same general result as Malcolm (see previous section), providing some more general examples like $19.\overline{9}=20$ and $399.\overline{9}=400$ (Marsh, 1742, p. 16). His justification of the equality is mainly topological, close to the argument given in (Dubinsky et al., 2005, p. 261-262).

\begin{figure}[h!]
\centering
\includegraphics[height=4cm]{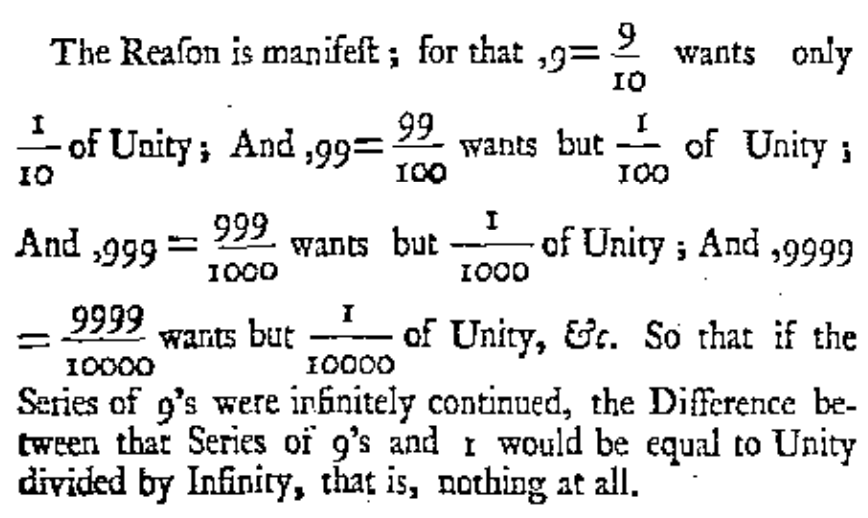}
\caption{Marsh's argument for $0.\overline{9}=1$}
\label{Marsh0}
\end{figure}

The equality $0.\overline{9}=1$ seems very natural to Marsh, as illustrates his first sentence of his proof (``The Reason is manifest''). Marsh is possibly the first author to show this equality in action at length, providing a lot of examples of its usefulness in subsequent calculations.  His first explicit case (Marsh, 1742, p. 36) is a calculation reproduced here in Figure \ref{Marsh1}, corresponding to the calculation $0.\overline{571428}+0.\overline{285714}+0.\overline{142857}=0.\overline{9}=1$. It appears among his first examples of computation with periodic expansions.

\begin{figure}[h!]
\centering
\includegraphics[height=3cm]{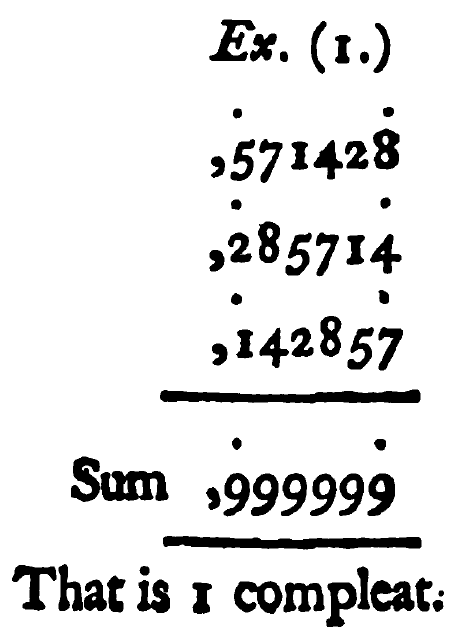}
\caption{A computation made by Marsh with periodic expansion, mentioning the identification of $0.\overline{9}$ with $1$.}
\label{Marsh1}
\end{figure}

Beforehand (Marsh, 1742, p. 32), we may find the same kind of identification reproduced in Figure \ref{Marsh2} but only in an indirect manner, in the calculation $0.9\overline{3}+0.7\overline{3}+0.2\overline{6}+0.0\overline{6}$.
\begin{figure}[h!]
\centering
\includegraphics[height=3cm]{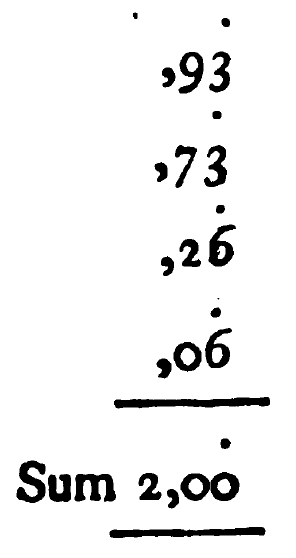}
\label{Marsh2}
\caption{Another calculation made by Marsh, with the implicit assumption that $0.\overline{9}=1$}
\end{figure}

The result is $2.0\overline{0}$ and not $1.9\overline{9}$ since the rule Marsh is following (given p. 30) is more or less equivalent to consider the ``circulate'' digits of each number as ninths parts, as Brown does (see Section \ref{GBrown}).

\subsection{An APOS interpretation}

The work of the previous authors (and some others) on decimal expansion of rational numbers can be interpretated in the framework of {\sc apos} theory. Here we focus on two aspects. The first one is the number obtained, with infinitely many digits. It may be understood as a process or object, here referred as $P_1$ and $O_1$. The second aspect is the understanding of the periodic part, which can also constitute a process or an object, $P_2$ and $O_2$. Besides the production of decimals by division ($P_1$ and $P_2$), we are interested more specifically in operations on objects $O_1$ and $O_2$ (which can also be interpretated as stages of {\sc apos} theory) as well as the equality $0.\overline{9}=1$ which allows to regard a fraction and its corresponding decimal expansion as equivalent, defining a rational number.

Wallis is at the stage of objects $O_1$ and $O_2$, furthermore asking for the length of the periodic parts. He is much interested in $O_2$, and not so much in $O_1$. He does not consider operations on periodic parts, and does not seem to remark the equality $0.\overline{9}=1$. He uses commas to circumscribe a periodic part, as in $0.803,571428,571428,57\text{\&c}$. Wallis has a more general notion of a number, mentioning the sexagesimal numeration system and, above all, real numbers (Wallis, chapter LXXXIX):

\begin{quote}But the concinnity which thus appears in the interminate Quotient of a Division, (the same numbers again returning in a continual Circulation ;) is not to be expected in like manner in the Extraction of Roots, (Square, Cubick, or of higher Powers.) For though the Surd Root may be continued by Approximation in Decimal parts, infinitely: Yet we have not therein the like recurrence of the numeral Figures in the same order, as in Division we had. As $\surd 2=1.41421356+$. Which yet hiders not but that this approximation may be safely admitted in practice; and if so supposed infinitely continued, must be supposed to equal the Root of that Surd number; as truly as 0.33333, \&c, infinitely, to equal $\frac{1}{3}$.
\end{quote}

The stage Totality is recognizable, with a specific + sign, meaning that the digits that come next are different from the first ones, providing the status of an object.

Brown seems to be at the stage $O_1$, operating on $O_1$ as if it were decimal numbers (by appending enough digits). We can also interpret what he does as a desencapsulation of $O_1$ allowing to operate (stage Action). This is facilitated by the fact that the numbers he deals with come from English monetary units, hence have periodic parts of length $1$, essentially $\overline{3}$ and $\overline{6}$ (deriving from $1/3$ and $1/12$). For example, he does not identify the periodic part of 908/19, only writing ``etc''. He seems to be at the $P_2$ process stage, since he does not operate directly on periodic parts, even if the circular powers identification (see section \ref{Representation1}) is recognizable (Brown, p. 11):

\begin{quote} Here you see the Decimal, for one penny is infinite, and yet you may limit it at any one of the reiterated Figures after Decimal thirds, or you may extend it as much further as you pleased.
\end{quote}

Cunn is at stage $O_1$, with an understanding of the equivalence between periodic decimal expansions and fractions, the use of integers of the form $99\cdots 99$ for the denominator, operations made on $O_1$ by algorithms showing the Process stage. He is also at stage $O_2$. He uses explicitely the word ``period'', and identifies $\overline{222222}$ with $\overline{2}$ (circular powers identification), with a specific notation for these, made of slashes delimiting the period (as in $8.59\!\!\!/35881\!\!\!/$ for $8.5\overline{935881}$). He even provides, for a multiplication, a period of length $24$. He gets the identity $0.\overline{9}=1$, refers to Wallis and makes use of his notation $+$ for the aperiodicity of square roots.

One of Hatton’s calculations makes it quite clear that identifying $0.\overline{9}$ and $1$ is beyond his scope (see Figure \ref{Hatton189}). Nevertheless, we can consider that he is at the stage $O_1$, since he succeedes in making computation with them. He is also at the stage $O_2$, with a specific notation. The Object is very clearly stated. However, he desencapsulates it into the Process when writing ``because the 4 is repeated'' to compute the product, then reencapsulates the Process into the Object with the notation $r6$ to indicates that six digits are to be repeated. He carries out some multiplications involving $1/3$, hence with periodic parts of length $1$. He seems quite close to Brown when he makes calculations with appending decimals rather than considering true periodic parts (apart from very simple cases like $3r+6r$).

Malcolm is at stages $O_1$ and $O_2$, with operations at the Process stage (with algorithms). He identifies clearly the equality $0.\overline{9}=1$, based on a geometric series (without details). The equivalence between fraction and periodic decimal expansion is explicit, with the periodic part as a numerator and $10^\ell-1$ as a denominator. Malcolm's notation for periodic part is made of a point over the initial and final digits of it. Also, Malcolm provides some examples of multiplication of periodic decimal expansions, but rather suggests to convert them into fractions (Malcolm, p. 483): ``it is much more tedious than the Multiplication of Finite Decimals, considering how easily the Finite Value of a Circulate is found ; and how easy it is to divide their Denominators''.

In his treatise, Marsh considers the full set of (positive) rational numbers and provide complete computation algorithms (referring to Wallis, Brown, Cunn and Malcolm). He is at stage $O_1$ and $O_2$ with his algorithms for standard operations (Process), but also writes a chapter on powers and roots. Multiplication reiterated could indicates that he is at the Object stage for operations. He also identifies the equality $0.\overline{9}=1$, with an explicit topological argument relying on a geometric series. He makes the equivalence between fractions and periodic expansions explicit, and he uses Malcolm's notations for periodic parts.

None of these authors set up a coherent and complete structure of the set of rational numbers from periodic decimal expansion, even if Cunn, Malcolm, and above them Marsh, were not far from it. To be more precise, consider the {\sc apos} notion of Schema, defined in (Arnon et al. p. 111) as ``a tool for understanding how knowledge is structured and its development through the learning process''. More specifically, we are interested in the schema for the set of rational numbers, in the three stages intra-$\Q$, inter-$\Q$ and trans-$\Q$. The same paper explains these stages as follows:

\begin{quote} In {\sc apos} theory, the Intra-stage of Schema development is characterized by a focus on individual Actions, Processes, and Objects in isolation from other cognitive items. At the Intra-stage, the student concentrates on a repeatable action or operation and may recognize some relationships or transformations among Actions on different components of the Schema. (p. 114)\end{quote}

\begin{quote} The Inter-stage is characterized by the construction of relationships and transformations among the Processes and Objects that make up the Schema. At this stage, an individual may begin to group items together and even call them by the same name. (p. 116)\end{quote}

\begin{quote} As a student reflects upon coordinations and relations developed in the Inter-stage, new structures arise. Through syntheses of those relations, the student becomes aware of the transformations involved in the Schema and constructs an underlying structure. This leads to development of the Schema at the Trans-stage. A critical aspect of the Trans-stage is development of coherence. Coherence is demonstrated by an individual’s ability to recognize the relationships that are included in the Schema and, when facing a problem situation, to determine whether the problem situation fits within the scope of the Schema. In some cases, the constructions involved in the mathematical definitions of a concept show coherence of the Schema; this means the individual is able to reflect on the explicit structure of the Schema and select from it the content that is suitable in solution of the problem. (p. 118)\end{quote}

Hence, we can specify the three stages of the Schema for rational numbers:

\begin{itemize}

\item Intra-$\Q$: periodic decimal expansions derive from fractions, with possibly the periodic part at a Process or Object stage, and also possibly Actions on these objects (operations). Wallis, Brown and Hatton are at this stage.

\item Inter-$\Q$: the fundamental relationship is the equality $0.\overline{9}=1$, which provides the link between the two equivalent representations of rational numbers. These ones can be properly defined as periodic decimal expansions, which are not subordinate to fractions anymore. The periodic parts are seen as objects, operations with algorithms can be at the Process stage. Notions of topology may arise to justify that $0.\overline{9}=1$. This stage is reached by Cunn and Malcolm.

\item Trans-$\Q$: the structure of periodic decimal expansions is coherent, with full equivalence with fractions, and operations as Objects and with their algebraic properties, defining the field $\Q$. With his treatise devoted to periodic decimal expansions and operations at the stage Object, Marsh is at this stage, even if not completely since he does not identify the structure of $\Q$ with decimal expansions. (The fact is that algebraic structures were not really considered for themselves at the time.)

\end{itemize}

It seems that such an analysis about periodic decimal expansions was never carried on to its end, since the reference Schema is mainly the one for $\R$ (as we can already see for Wallis). In the next section, the aim is to set up a mathematical framework that could constitute the basis for a Schema trans-$\Q$.

\section{Combinatorial definitions for sets isomorphic to $\Q$}\label{Combi}

Here and in the next sections, we are interested in more strictly mathematical aspects of circular words in base $b$. We wish to investigate what can be done with a purely combinatorial definition of the set $\Q$ without making use of fractions. Our aim is to define two set (hereafter named ${\cal Q}_{\text{WCP}}$ and ${\cal Q}_{\text{DC}}$) that will be eventually proved to correspond to $\Q$. Both definitions derive from the characterization of rational numbers given by Theorem \ref{Main}, namely: for any integer $b\geqslant 2$, a number is rational iff its $b$-expansion is ultimately periodic.

The definition of the two sets ${\cal Q}_{\text{WCP}}$ and ${\cal Q}_{\text{DC}}$ are quite similar, but are suited for different purposes. The first one remains close to the usual perception of $b$-expansion, the second one is better for theoretical reasoning. Ultimately, such representations will allow us to define the field $\Q$ (Sections \ref{Add} and \ref{Mult}), and to provide a new and quite simple proof that algebraic integers are either integers or irrational numbers (Section \ref{ThIrr}).

\subsection{Circular words}\label{CircW}

Let $b\geqslant 2$ be a fixed integer. A $b$-expansion of a number $x$ is a codage of $x$ by a sequence, called a {\em word}, of elements of the {\em alphabet} ${\cal A}:=\{0,1,\ldots,\beta\}$, where $\beta:=b-1$. A finite word $W$ is generically written $w_0w_1\cdots w_{\ell-1}$ (with $w_i\in{\cal A}$ for all $i$), where $\ell$ is the {\em length} of $W$, also written $|W|$. For $W'=w'_0\cdots w'_{\ell'}$, we define $WW'$ as the concatenation of $W$ and $W'$, that is: $WW'=w_0\cdots w_\ell w'_0\cdots w'_{\ell'}$. Defining $W^1:=W$, we also put, for any $n\geqslant 2$, $W^n:=WW^{n-1}$.

Occasionaly there will be some ambiguity with the notation for exponents, but the context will make things clear. For example, the expression $\beta^\ell$ will always stand for the concatenation of $\ell$ copies of the single-letter word $\beta$ (and never for the value $\beta$ to the $\ell$-th power), whereas $b^n$ will always denote the usual power of the natural number $b$.

The application ${\displaystyle N(w_0\ldots w_{\ell-1}):=\sum_{i<\ell}w_ib^{\ell-1-i}}$ defines a one-to-one correspondence between the set of finite words with $w_0\neq 0$ and the set $\N^*$. Such an application corresponds to the usual writing in base $b$ (in which the rightmost letter $w_{\ell-1}$ corresponds to units). To avoid cumbrous notations, we will frequently confuse $W$ and $N(W)$ in the sequel. Again, the context will make things clear.

We already encountered {\em infinite periodic words} in the previous sections, denoted by $\overline{W}$ (as in $0.\overline{9}$ or $0.\overline{873}$). A slightly distinct notion is the notion of {\em circular word of length $\ell$}, a word $\widetilde{W}$ whose letters are indexed by $\Z/\ell\Z$ instead of $\{0,1,\ldots, \ell-1\}$. Intuitively speaking, a circular word it is a finite word in which its final letter is followed by its initial one (indexed by $0$, so one may speak of {\it dotted} circular words for the sake of precision). The set of circular words of length $\ell$ is denoted by $\widetilde{{\cal A}^\ell}$, and the set ${\displaystyle\bigcup_{\ell\geqslant 1}\widetilde{{\cal A}^\ell}}$ of all circular words on the alphabet ${\cal A}$ is written $\widetilde{\cal A}$. The {\em shift} $\sigma$ on $\widetilde{{\cal A}}$ is the bijection such that $\sigma(\widetilde{w_0\ldots w_{\ell-1}})=\widetilde{w_1\ldots w_{\ell-1}w_0}$ for any circular word of length $\ell$.

Despite its natural appearance in decimal expansion of rational numbers, the study of circular words seems to be very recent (Rittaud \& Vivier, 2012b, 2011; Rittaud, to appear), apart from its intuitive utilization.\footnote{Wallis already coins the term ``circulation''; Marsh uses indifferently the terms ``circulant'' and ``repetend''; between Wallis and Marsh, William Jones (Jones, 1706, p. 104-105) talked of ``circulating figures'', Samuel Cunn (1714, p. 61)  of ``circulating numbers'', Alexander Malcolm (1730, p. 150) of ``circulating decimals''. Some years after Marsh, John Robertson (Robertson, 1769) speaks of ``circulating fractions''.} Here is a classical and interesting application:

\begin{theorem}[Fermat's little theorem]\label{Fermat} Let $p$ be a prime number. For any integer $b$, we have $b^{p}\equiv b\ (\bmod\ p)$.
\end{theorem}

\begin{proof} Since $p$ is prime, for any $\widetilde{W}\in\widetilde{{\cal A}^p}$ not of the form $\widetilde{w^p}$, we have that $\widetilde{W}$, $\sigma(\widetilde{W})$, \ldots, $\sigma^{p-1}(\widetilde{W})$ are different circular words. Hence, $\widetilde{{\cal A}^p}$ splits into subsets made of exactly $p$ elements (the equivalence classes under the equivalence relation $\widetilde{W}\sim\widetilde{W'}$ iff $\widetilde{W'}=\sigma^k(\widetilde{W})$ for some $k$, except for the words of the form $\widetilde{w^p}$), plus the subset $\{\widetilde{w^p},\ w\in{\cal A}\}$, which contains exactly $b$ elements. Therefore, the cardinality of $\widetilde{{\cal A}^p}$, equal to $b^p$, is also of the form $kp+b$, where $k$ is some integer. We thus have $b^p\equiv b\ (\bmod\ p)$.\end{proof}

Note that, apart from the primality of $p$, the proof relies on combinatorics, not on arithmetic. A way to generalize the theorem is to consider circular words with some combinatorial constraints. For example, let ${\cal A}:=\{0,1\}$, and consider the set of circular words of length $\ell$ in which the subword $11$ does not appear (note that $11$ appears in the word  $\widetilde{10^{\ell-2}1}$ because of the circular structure). It can be shown that its cardinality is given by the {\em Lucas sequence} $(L_\ell)_\ell$ defined by $L_1=1$, $L_2=3$ and $L_\ell=L_{\ell-1}+L_{\ell-2}$ (see Rittaud \& Vivier, 2012b). The same proof as before then gives the following variant of Fermat's Little Theorem: for any prime number $p$, we have $L_p\equiv 1\ (\bmod\ p)$. This result can, of course, be generalized to other combinatorial constraints\footnote{A lot of what is presented here can probably be generalized to $b$-expansions for algebraic values of $b$. (The previous example corresponds to the case $b=(1+\sqrt{5})/2$.) Nevertheless, the extension of the theory for these values is in no way trivial and still a work in progress. For example,  Theorem \ref{StructGpe} does not hold for $b=(1+\sqrt{5})/2$.}.

\subsection{The set ${\cal Q}_{\text{\rm WCP}}$ of word-circular-point representation of rational numbers}\label{Representation1}

This set is in some sense the most natural one, and from a teaching perspective the simplest one. As presented in Section \ref{Histoire}, it is, at least in an implicit way, the representation chosen by English authors of the {\sc xviii}\textsuperscript{th} century involved in the study of periodic decimal expansion for practical arithmetics.

Consider the rational number whose decimal expansion is $24.837\overline{56}$. We will say here that its {\it word-circular-point representation} (WCP) is the triple $(24837,\widetilde{56},-3)$. In this triple, the part $24837$ corresponds to the aperiodic part of the expansion, $\widetilde{56}$ corresponds to the periodic part, and $-3$ localizes the position of the decimal point (by counting the number of digits between it and the beginning of the periodic part, counted negatively if the decimal point lies in the aperiodic part and positively otherwise). Also, to get negative rational numbers, we would need to symmetrize the set, which can be done by defining quadruples $(s,W,\widetilde{P},c)$ with $s\in\{+,-\}$. This would be quite cumbrous, so we will not consider it in the following, but in Section \ref{AbelianGroupQ} where it is needed to get the group structure.

It is easy to prove that the rational number that corresponds to the given triple $(W,\widetilde{P},c)$ is $b^c\big(W+P/\beta^{|P|}\big)=b^c(W+0.PPP\ldots)$. To make this application bijective, several identifications are to be made to take into account that several expressions of the form $(W,\widetilde{P},c)$ correspond to the same rational number. These identifications are:

\begin{enumerate}

\item the leading zeroes identification: $(W,\widetilde{P},c)\equiv (0W,\widetilde{P},c)$;

\item the circular powers identification: $(W,\widetilde{P},c)\equiv(W,\widetilde{P^k},c)$ for any $k\in\N^*$;

\item the ``$0.999\ldots=1$'' identification: $(W,\widetilde{\beta},c)\equiv(W+1,\widetilde{0},c)$;

\item the shift identification: $(W,\widetilde{P},c)\equiv(Wp_0,\sigma(\widetilde{P}), c-1)$, where $p_0$ is the initial letter of $\widetilde{P}$;

\end{enumerate}

\begin{definition} The set ${\cal Q}_{\text{\rm WCP}}$ is the set of triples $(W,\widetilde{P},c)$ quotiented by these four identifications. 
\end{definition}

Most of the English authors of the {\sc xviii}\textsuperscript{th} state the circular powers and the shift identifications, mainly for the purpose of addition (see Section \ref{AbSt}).

The three first identifications are somewhat inescapable: the first one provides a rule for the aperiodic part, the second one a rule for the periodic part, and the third one articulates the link between the two, link without which the structure would reduces to a direct product. There is still a gap between the two firsts, easily  accepted  at an elementary level, and the third one, much more difficult to accept (see  Section \ref{0999Additive} for a way to make the third one ``natural''; note also that we will make use of it only when the addition of rational numbers is defined). As for the shift identification, it does not rely on a fundamental structure, it is more a technical identification, which is less satisfactory in a theoretical meaning. This inconvenience will be overcome by our second construction (section \ref{Representation2}), in which there will be no need for such a shift identification.

Even if we will not elaborate on this later, observe that the integer $c$ is more important for addition than for multiplication in ${\cal Q}_{\text{\rm WCP}}$. Indeed, addition of numbers given on a $b$-expansion form requires a clear positioning of the digits, whereas multiplication does not. As can be checked, Marsh's algorithm for multiplication (Section \ref{JMarsh}) does not fundamentally need a value $c$. Algebraic properties of the set made of pairs $(W,\widetilde{P})$ has some historical roots since, as shown by historians like Fran\c{c}ois Thureau-Dangin (1930, p. 117) and Christine Proust (2007, p. 249-251), such a numeration system without position was the underlying mathematical structure in use in the Babylonian sexagesimal numeration system, four millenia ago.

\subsection{The set ${\cal Q}_{\text{\rm DC}}$ of  decimal-circular representation of rational numbers}\label{Representation2}

This second representation is interesting in that it can be seen as a more natural extension of the set $\D_b$ of $b$-decimal numbers ({\it i.e.} the set of all rational numbers which can be written on the form $\delta=u\cdot b^e$ with $u\in\Z$ and $e\in\Z$).

First, we  define a ring ${\cal D}_b$ isomorphic to the ring $\D_b$ in a combinatorial way similar to the previous ones. An element $\delta$ of ${\cal D}_b$ is a finite word $W$ on the alphabet ${\cal A}$ together with an integer $c$ with $0\leqslant c\leqslant |W|$ ($c$ corresponds to the place of the ``decimal'' point, with $c=0$ for $c$ after the rightmost letter of $W$) and a sign $s\in\{+,-\}$. The necessary and sufficient identifications that make ${\cal D}_b$ isomorphic to $\D_d$ are therefore: $(s,W,c)=(s,0W,c)=(s,W0,c+1)$.

Up to some more identifications, the set ${\cal Q}_{\text{\rm DC}}$ is then defined as ${\cal D}_b\times\widetilde{{\cal A}}$. To understand the way it is done, consider again the rational number $r=24.837\overline{56}$. Its representation by a pair $(\delta,\widetilde{P})$ consists, in some sense, in forcing its periodic part to start right after the decimal point by writing the number as the sum $24.181+0.\overline{65}$, so the number is represented by the pair $(24.181, \widetilde{65})$. This is the {\em decimal-circular} representation.

The $(\delta,\widetilde{P})$ corresponds to the rational number $\delta+P/\beta^{|P|}=\delta+0.PPP\ldots$. As for the identifications we need to make ultimately ${\cal Q}_{\text{\rm DC}}$ isomorphic to $\Q$ as fields, they appear to be more natural than those for ${\cal Q}_{\text{\rm WCP}}$. The leading zeroes identification now derives from the preliminary construction of ${\cal D}_b$ and, most importantly, the shift identification is not required anymore. Moreover, there is no need for an {\em ex post} symmetrization since ${\cal D}_b$ is already a group. The remaining identifications are:

\begin{itemize}

\item the circular powers identification: $(\delta,\widetilde{P})=(\delta,\widetilde{P^k})$ for any $k\in\N^*$;

\item the ``$0.999\ldots=1$'' identification: $(\delta,\widetilde{\beta})\equiv(\delta+1,\widetilde{0})$, where $\delta=1$ is defined in a standard way.

\end{itemize}

\begin{definition} The set ${\cal Q}_{\text{\rm DC}}$ is the set of pairs $(\delta,\widetilde{P})$ quotiented by these identifications.
\end{definition}

In a teaching perspective, the main inconvenient of ${\cal Q}_{\text{\rm DC}}$ is that the decimal number $\delta$ may strongly differ from the usual aperiodic part of the represented number. For example, the DC representation of $2.14444\ldots$ is $(1.7,\widetilde{4})$, hence its integer part is not  equal to the integral part of its $\delta$. It can even occur in some cases that $\delta$ is negative whereas $r$ is positive, as the example of $0.47777\ldots=(-0.3,\widetilde{7})$ shows.

Apart from this inconvenience and the cumbreness it produces for ordering considerations (see Section \ref{OrderedSet}), ${\cal Q}_{\text{\rm DC}}$ appears to be considerably more tractable than ${\cal Q}_{\text{\rm WCP}}$ in most aspects.

\subsection{Order on ${\cal Q}_{\text{\rm WCP}}$ and ${\cal Q}_{\text{\rm DC}}$}\label{OrderedSet}

Order on ${\cal Q}_{\text{\rm WCP}}$ and ${\cal Q}_{\text{\rm DC}}$ are slightly difficult to be defined on a proper way, since the identifications forced by $\equiv$ makes it dfficult to provide a definition simply from triples $(W,\widetilde{P},c)$ and $(\delta,\widetilde{P})$.

First, we can define the lexicographical order $\leqslant$ on the set of finite words on the alphabet ${\cal A}$ in the following way: for $W=w_\ell\cdots w_1$ and $W'=w'_{\ell'}\cdots w'_1$ with $\ell<\ell'$, replace first $W$ by $0^{\ell'-\ell}W$ to ensure both words have the same length, then let $k$ be the biggest index for which $w_k\neq w'_k$ (if any; otherwise $W=W'$). Then, $W$ and $W'$ are in the same order as $w_k$ and $w'_k$ are in ${\cal A}$.

In $\widetilde{\cal A}$, we can define a lexicographical order as well: to compare two circular words $\widetilde{P}$ and $\widetilde{P'}$, we use the identifications ${\widetilde{P^k}}=\widetilde{P}$ and ${\widetilde{P'^{k'}}}=\widetilde{P'}$ to get two circular words of the same length, then use the lexicographical order to decide which one is bigger than the other.

In ${\cal Q}_{\rm WCP}$, to compare $(W, \widetilde{P}, c)$ and $(W', \widetilde{P'},c')$, we first use the circular powers and shift identifications to transform $c'$ into $c$ and to have $|\widetilde{P}|=|\widetilde{P'}|$, then the leading zeroes identification to get $|W|=|W'|$. Since the sets ${\cal A}^\ell$ and $\widetilde{{\cal A}^\ell}$ are ordered by the lexicographical order (both denoted by $<$), a first natural order on ${\cal Q}_{\rm WCP}$ is the {\em semiotic} one, here denoted by $\prec$ (and $\preccurlyeq$):
\[(W,\widetilde{P},c)\prec(W',\widetilde{P'},c)\Longleftrightarrow\left\{\begin{array}{l}
W<W'\ \mbox{\em or}\\
W=W'\ \mbox{and}\ \widetilde{P}< \widetilde{P'},\end{array}\right.\]

\noindent the binary relation $\preccurlyeq$ being defined in the same way, only replacing $\widetilde{P}< \widetilde{P'}$ by $\widetilde{P}\leqslant \widetilde{P'}$.

Such a semiotic order suggests, as scholars and students often believe, that $0.\overline{9}< 1$ (since $(0,\widetilde{9},0)\prec (1,\widetilde{0},0)$). To take into account the ``$0.999\ldots=1$'' identification, the order $\leqslant$ we wish to define is the following one:
\[(W,\widetilde{P},c)\leqslant(W',\widetilde{P'},c)\Longleftrightarrow\begin{cases}(W,\widetilde{P},c)\preccurlyeq(W',\widetilde{P'},c)&\mbox{\em or}\\ W'=W+1,\ \widetilde{P}=\widetilde{0}\ \mbox{\rm and }\widetilde{P'}=\widetilde{9}&\mbox{\em or}\\  W=W'+1,\ \widetilde{P}=\widetilde{9}\ \mbox{\rm and }\widetilde{P'}=\widetilde{0}. \end{cases}\]

Now, let us compare the DC representations of $x=(\delta,\widetilde{P})$ and $x'=(\delta',\widetilde{P'})$. By the circular powers identification, we may assume that $|\widetilde{P}|=|\widetilde{P'}|=\ell$. By the equalities $x=\delta+P/\beta^\ell$ and $x'=\delta'+P'/\beta^\ell$, we easily get that $x<x'$ iff $b^\ell(\delta-\delta')<\delta-\delta'+P'-P$ (where $P$ and $P'$ are to be understood as integers). This way to present the inequality allows to minimize the non-combinatorial calculations to be made too compare $x$ and $x'$ (recalling that multiplying by $b^\ell$ corresponds to a shift).

\section{The abelian groups $({\cal Q}_{\text{\rm WCP}},+)$ and $({\cal Q}_{\text{\rm DC}},+)$}\label{Add}

\subsection{Abelian structure on circular words}\label{AbSt}

Circular words of length $\ell$ can be added in the same way as for usual $b$-expansion of integers, except when, as in the right example in Figure \ref{ExSum}, the sum of the leftmost digits ($7+5$) produces a carry, which has to be put on the rightmost place (thus changing the $6$ into a $7$).

\begin{figure}[h!]
\centering
\includegraphics[height=3cm]{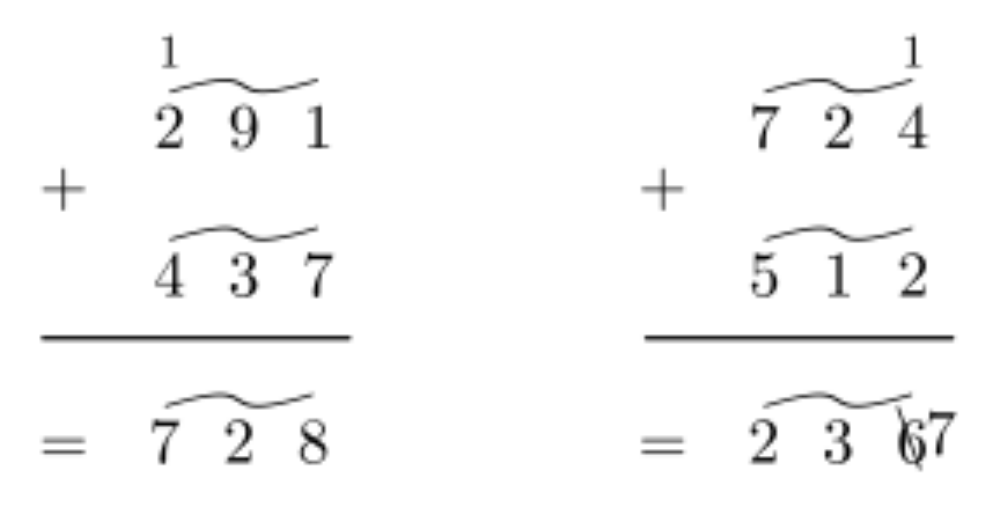}
\caption{Examples of summations of circular words.}
\label{ExSum}
\end{figure}

The addition in $\widetilde{{\cal A}^\ell}$ is associative, commutative and admits $\widetilde{0^\ell}$ as a neutral element. Unfortunately, this addition does not make $\widetilde{{\cal A}^\ell}$  a group, since none of its elements admits an inverse element (apart for $\widetilde{0^\ell}$ itself). To get an abelian group, the following observation can be made: for any $\widetilde{W}\in\widetilde{{\cal A}^\ell}\setminus\left\{\widetilde{0^\ell}\right\}$, we have $\widetilde{W}+\widetilde{\beta^\ell}=\widetilde{W}$. In other words, $\widetilde{\beta^\ell}$ is an ``almost neutral element'' of $(\widetilde{{\cal A}^\ell},+)$ (see also (Rittaud \& Vivier, 2012a). If we can make it a true neutral element, then $\widetilde{{\cal A}^\ell}$ becomes a group: indeed, the opposite of any $\widetilde{W}$ is the circular word in which each letter $w$ of $\widetilde{W}$ is replaced by $\beta-w$.

The idea, behind which lies for a part the equality $0.999\ldots=1$, is therefore to identify $\widetilde{0^\ell}$ and $\widetilde{\beta^\ell}$, thus defining a new set, ${\cal G}_\ell$, quotient of $\widetilde{{\cal A}^\ell}$ under the single equivalence $\widetilde{0^\ell}=\widetilde{\beta^\ell}$. On this new set the addition is well-defined (since we have $\widetilde{0^\ell}+\widetilde{\beta^\ell}=\widetilde{\beta^\ell}=\widetilde{0^\ell}$) and makes ${\cal G}_\ell$ an abelian group. Since ${\cal G}_\ell$ is also monogenetic (a generator is $\widetilde{0^{\ell-1}1}$) and contains $b^\ell-1$ elements, we have:

\begin{theorem}\label{StructGpe} The abelian group ${\cal G}_\ell$ is isomorphic to $\Z/(b^\ell-1)\Z$.
\end{theorem}

Note that, in ${\cal G}_\ell$, the shift corresponds to the multiplication by $b$. As a consequence, any subgroup ${\cal H}$ is shift-invariant (that is: $\sigma({\cal H})={\cal H}$).\\

Now, put ${\displaystyle{\cal G}=\bigcup_{\ell\geqslant 1}{\cal G}_\ell}$, and let ${\cal G}^*$ be the quotient set $\widetilde{{\cal G}}/\equiv$ defined by the equivalence relation that identifies a circular word with all its nontrivial powers, that is:
\[\widetilde{W}\equiv\widetilde{W'}\Longleftrightarrow \widetilde{W^m}=\widetilde{W'^n}\mbox{ for some positive $m$ and $n$.}\]

Thus, an addition in ${\cal G}^*$ can be easily derived from the addition in ${\cal G}_\ell$. For example, we have  $\widetilde{54}+\widetilde{627}=\widetilde{545454}+\widetilde{627627}=\widetilde{173082}$, and it is easily proved that such a definition is consistent.

\begin{theorem} The set ${\cal G}^*$ equipped with this addition is a abelian group. Any finite subgroup of it is monogenetic. If $p$ is a prime number, then the set ${\cal P}_p$ made of all circular words of order $p$ (plus the neutral element) is a finite subgroup isomorphic to $\Z/p\Z$ iff $b$ and $p$ are mutually primes.
\end{theorem}

\begin{proof} The first part of the theorem is trivial. Consider a finite subgroup of ${\cal G}^*$, say ${\cal H}$. By identification of circular words with their powers, it is possible to find an $\ell>0$ such that any element of ${\cal H}$ has a representative in $\widetilde{{\cal G}_\ell}$. Hence, from Theorem \ref{StructGpe} we deduce that ${\cal H}$ is monogenetic.

Now, for any $\ell>0$, consider the set ${\cal L}_\ell$ of all the elements of ${\cal G}^*$ of order $p$  which admits a representative of length at most $\ell$. As before, there exists $\ell'>0$ such that all the elements of ${\cal L}_\ell$ has a representative in ${\cal G_{\ell'}}$. Hence, completing the set with the neutral element of ${\cal G_{\ell'}}$, It is easily proved that we get a subgroup of ${\cal G_{\ell'}}$. Again by Theorem \ref{StructGpe}, ${\cal L}_\ell$ is therefore either reduced to the neutral element or equal to $\Z/p\Z$.

The last remaining thing to do is to prove that, for some big enough $\ell$, the set ${\cal L}_\ell$ is nontrivial iff $b$ and $p$ are mutually primes. Two different presentations of the proof can be given: a purely algebraical one and a presentation with the help of common fractions. Here we give both of them, the first one being more in line with the genral way we construct our sets, the second being an unexpected application of Proposition \ref{UnSurv}. 

Let us first provide the algebraical proof. First, assume $b$ and $p$ without any common divisor. By Fermat's little theorem, there exists $\ell>0$ and $v\in\N$ such that $b^\ell-1=vp$, hence $v$ is of order $p$ in $\Z/(b^\ell-1)\Z$ which is isomorphic to ${\cal G}_\ell$ by Theorem \ref{StructGpe}, so we are done. Now assume $b=kp$ and that there exists $g\in{\cal G}_\ell$ of order $p$. By the group isomorphism, we can therefore find $v\in\Z$ such that $vp\equiv  0 \bmod (b^\ell-1)$, so we can find an integer $c$ such that $vp=c(b^\ell-1)$. Since $p$ and $b^\ell-1$ are mutually prime and $p$ divides $c(b^\ell-1)$, $p$ divides $c$, so $c=pc'$ and $v=c'(b^\ell-1)$, so $g$ cannot be of order $p$.

 Now for the second presentation of the proof. Again, assume first that $b$ and $p$ have no common divisor.
Then, the $b$-expansion of $1/p$ is purely periodic by Proposition \ref{UnSurv}, and the circular word associated to its period provides an element of ${\cal L}_\ell$ (for big enough $\ell$). Conversely, assume that ${\cal L}_\ell$ has a nontrivial element $\widetilde{W}$, and consider the number $x=0.\overline{W}$. We then have that $n:=px\in\N^*$, so $n/p=x$ has a purely periodic $b$-expansion, hence $p$ and $b$ are mutually primes by Proposition \ref{UnSurv}.\end{proof}

\subsection{The commutative ring $\D_b$ and its action on circular words}\label{RingD}

The standard action of $\Z$ on $\widetilde{{\cal A}}$ defined by the addition with ``circular carry'' extends naturally to an action of the ring $\D_b$, which therefore makes $\widetilde{{\cal A}}$ a module over $\D_b$, hence also on ${\cal D}_b$ (as defined in Section \ref{Representation2}). Such an extension goes as follows: let $\delta=u\cdot b^e\in\D_b$, and let $(s,W,c)$ be its counterpart in ${\cal D}_b$. Let $\widetilde{P}\in {\cal G}^*$. . Then, $\delta\widetilde{P}$ is defined as $\sigma^{c}(s\cdot W\widetilde{P})$ (or, equivalently, as $\sigma^e(u\widetilde{P})$) where $\sigma$ is the shift operator on circular words and $W$ stands for its corresponding integer. Observe that, in this expression, $\sigma^c$, $s$ and $W$ commute (as well as $u$ and $\sigma^e$).

\subsection{Abelian structure on $({\cal Q}_{\text{\rm WCP}},+)$ and $({\cal Q}_{\text{\rm DC}},+)$}\label{AbelianGroupQ}

The simplicity of the addition in ${\cal G}_\ell$ cannot remain unchanged when considering ultimately periodic $b$-expansion of rational numbers, since we have to deal with a possible carry. Therefore, we need some complement to get the additive structure from the additive structure of circular words. Defining binary operations is a task simpler in ${\cal Q}_{\text{\rm DC}}$ than in ${\cal Q}_{\text{\rm WCP}}$. Nevertheless, because of the distinct interests of these representations, we will do it for both of them.

For any word $W$ and any $\ell>0$, define $q_\ell(W)$ as the quotient of the Euclidean division of $W$ by $\beta^\ell$. Most of the time we will simply write $q(W)$ instead, under the following type of assumption: let $\widetilde{P_1}$ and $\widetilde{P_2}$ be two circular words, of lengths $\ell_1$ and $\ell_2$. Then, $q(P_1+P_2)$ stands for $q_\ell(Q_1+Q_2)$, where $\ell$ is a common multiple of $\ell_1$ and $\ell_2$ and $\widetilde{Q_i}=\widetilde{P_i^{\ell/\ell_i}}$.

With such definition (which in fact corresponds to the definition of a carry), we can define addition in  ${\cal Q}_{\text{\rm WCP}}$ and in  ${\cal Q}_{\text{\rm DC}}$. As announced, it is immediate in the latter one:

\[\Big(\delta,\widetilde{P}\Big)+\Big(\delta',\widetilde{P'}\Big):=\Big(\delta+\delta'+q(P+P'),\widetilde{P}+\widetilde{P'}\Big).\]

As regards  ${\cal Q}_{\text{\rm WCP}}$, we have to suppose that the decimal points are at the same place to add conveniently the two numbers:
\[\Big(W,\widetilde{P}, c\Big)+\Big(W',\widetilde{P'},c\Big)=\Big(W+W'+q(P+P'),\widetilde{P}+\widetilde{P'},c\Big),\]
\noindent and to add two numbers with different values for $c$ we need to make use of the shift identification. Such a theoretical cumberness is a quite strong argument in favour of the use of  ${\cal Q}_{\text{\rm DC}}$ instead, even if the shift identification often remains an easy task in practice.

In the case ${\cal Q}_{\rm WCP}$, as is briefly indicated in Section \ref{Representation1}, elements are in fact quadruples $(s,W,\widetilde{P},c)$ with $s\in\{+,-\}$. The general definition of the addition has to be given also for two quadruples of different signs, say $(+,W,\widetilde{P},c)+(-,W', \widetilde{P'},c)$. By the ``0.999\ldots=1'' identification, we can assume that none of the quadruples contains the circular word $\widetilde{9}$, hence it makes sense to compare $\widetilde{P}$ and $\widetilde{P'}$ with the lexicographical order. Also, by the shift identification, we can suppose $c=0$. We then have
\[(+,W,\widetilde{P},0)+(-,W', \widetilde{P'},0)=\begin{cases}(+,W-W',\widetilde{P}-\widetilde{P'},0)&\text{if $W\geqslant W'$ and $\widetilde{P}\geqslant\widetilde{P'}$}\\
(+,W-W'-1,\widetilde{P'}-\widetilde{P},0)&\text{if $W>W'$ and $\widetilde{P}<\widetilde{P'}$}\\
(-,0,\widetilde{P'}-\widetilde{P},0)&\text{if $W=W'$ and $\widetilde{P}<\widetilde{P'}$}\\
\end{cases}\]
\noindent the other cases being obtained by a simple symmetrization.

There is no need for such exhaustion of cases in the context of ${\cal Q}_{\rm DC}$, which is another argument in favor of its use. Nevertheless, the opposite of an element is less simply written: $-(\delta,\widetilde{P})=(-\delta-1,-\widetilde{P})$. To avoid this quite counterintuitive minus $1$, a possibility would be to define the circular words $\widetilde{P}$ in an alphabet ${\cal A}$ symmetric around $0$ (and reconsider the theory accordingly). The problem is then that the expression of numbers would become even more different from the common one, and also that this symmetrization could work only for a odd value of the base $b$.

Eventually, what precedes leads to the following

\begin{theorem} ${\cal Q}_{\text{\rm WCP}}$ and ${\cal Q}_{\text{\rm DC}}$ are abelian groups.

\end{theorem}

\subsection{0.999\ldots=1 from the additive structure}\label{0999Additive}

The additive structure defined on section \ref{AbelianGroupQ}, either on the triples $(W,\widetilde{P},c)$ or on the pairs $(\delta,\widetilde{P})$, provides a powerful argument to make the equality $0.999\ldots=1$ sensible.  Indeed, observe that, under the rules defining the addition, we have, for any $\widetilde{P'}\neq\widetilde{0^\ell}$:

\[(W+W'+1,\widetilde{P'},c)=\left\{\begin{array}{l}(W,\widetilde{\beta},c)+(W',\widetilde{P'},c)\\(W+1,\widetilde{0},c)+(W',\widetilde{P'},c)\end{array}\right.\]
\noindent and
\[(\delta+\delta'+1,\widetilde{P'})=\left\{\begin{array}{l}(\delta,\widetilde{\beta})+(\delta',\widetilde{P'})\\(\delta+1,\widetilde{0})+(\delta',\widetilde{P'})\end{array}\right..\]

Therefore, the cancellation property ($a+x=b+x\Rightarrow a=b$) strongly supports the identifications $(W,\widetilde{\beta},c)=(W+1,\widetilde{0},c)$ and $(\delta, \widetilde{\beta})=(\delta+1,\widetilde{0})$, which corresponds to the equality  $0.999\ldots=1$.

It is an important fact that the identification between $0.999...$ and 1 is required for the additive structure of the set of rational numbers. It comes from algebra, and not from analysis for the completness of $\mathbb{R}$ as often mentioned, even if this identity has implication for continuum of the real numbers set.

It is worth mentioning that in the "equation method" to prove that $0.\overline{9}=1$ (in which $0.\overline{9}$ is written as $x$, to write that $10x=9+x$ then $9x=9$, hence $x=1$), the cancellation property is used as well. Hence, the process in itself is directly linked to the assumption that the property is legitimate.

\subsection{The $b$-adic case}\label{LeCasbadique}

For the record, let us also mention briefly that the previous constructions also allow to define periodic $b$-adic numbers, which are also the ones that are rational in $\Q_b$ by a classical theorem analogous to Theorem \ref{Main}. Instead of $(W,\widetilde{P},c)$ or $(\delta, \widetilde{P})$, we may write $(\widetilde{P},W,c)$ and $(\widetilde{P},\delta)$, since $b$-adic numbers are those whose expansion in base $b$ has infinitely many digits to the left. The identifications to be made are the same as in the case of rational numbers, with the only exception of the ``$0.999\ldots=1$'' one, which has to be replaced by a different one, namely the ``$\ldots 999=-1$'' identification, that writes $(\widetilde{\beta},W,c)\equiv (\widetilde{0},W-1,c)$ for the WCP representation, and $(\widetilde{\beta},\delta)\equiv(\widetilde{0},\delta-1)$ for the DC one.

\section{Multiplicative structure}\label{Mult}

\subsection{Preliminaries}\label{CircMult}

Theorem \ref{StructGpe} could provide us a notion of multiplication for circular words. A possible presentation of it is to consider the set of circular words of length $\ell$ as the quotient ring $\Z[X]/(X^\ell-1,X-b)$. The point is that such a multiplication does not correspond to multiplication of periodic parts of $b$-expansion of rational numbers. 

Let us consider two circular words $\widetilde{W}$ and $\widetilde{X}$ of lengths $\ell$ and $\ell'$ respectively. Their product $\widetilde{W}\times\widetilde{X}$ may be defined as the circular word $\widetilde{Y}$ of length $n$ such that $(W/\beta^\ell)\times(X/\beta^{\ell'})=Y/\beta^n$, provided that the existence of $n$ is ensured\footnote{This idea may be indirectly recognized in some multiplications made by Marsh (1742), where he considers only expressions like $0.\overline{M}$ to be multiplied together. Nevertheless, since Marsh does not consider circular words independently of numbers, he does not operate directly with them.}. For example, since $0.\overline{12}\times 0.\overline{4}=(12/99)\times(4/9)=53872/999999$, we may define $\widetilde{12}\times\widetilde{4}$ as $\widetilde{053872}$. The only contentious point is that, to make this definition sensible, $\widetilde{0}$ and $\widetilde{9}$ should not be identified anymore, since $\widetilde{0}$ is absorbing whereas $\widetilde{9}$ is neutral.

Not only we know that no multiplication in ${\cal G}_\ell$ can therefore correspond to the multiplication of periodic parts of rational numbers, but an important difficulty for concrete multiplication of circular words also derived from the fact that, as illustrated by Proposition \ref{CasPartFauxFois},  the length increases surprisingly fast. Attempts made by English authors like those mentioned in Section \ref{Histoire} to provide extensive algorithms were valuable, especially Marsh's, but remained cumbrous since none of them endorsed a theoretical point of view which is helpful to get both a complete algorithm and a rigorous proof of its validity.  

In Section \ref{LongueurProduit}, we provide a generalization of Proposition \ref{CasPartFauxFois} for the existence and smallest value of $n$ for the length of $Y$, knowing those of $W$ and $X$ (under the previous notation). Such an $n$ allows to define $\widetilde{W}\times\widetilde{X}$ as the circular word $\widetilde{Y}$ of length $n$ such that $Y=WX\beta^n/(\beta^{\ell}\times\beta^{\ell'})$. Then, an explicit expression of $\beta^n/(\beta^{\ell}\times\beta^{\ell'})$ is given in Section \ref{AlgoMultCirc} for $\ell=\ell'=1$, obtained as a consequence of the proof of Theorem  \ref{99FauxFois99} (and shown as being a simplifying assumption without loss of generality). Eventually, Section \ref{AlgoMult} provide an algorithm for multiplication in $\Q$ from circular words.

\subsection{Length of a product}\label{LongueurProduit}

This section is devoted to the following result, that extends what we proved in Proposition \ref{CasPartFauxFois} in the particular case of $b=10$ and $\ell=\ell'=2$.

\begin{theorem}\label{99FauxFois99} Let $\ell$ and $\ell'$ be two positive integers. The smallest positive integer $n$ fort which $b^n-1$ is divided by $(b^\ell-1)(b^{\ell'}-1)$ is $n=(b^{\gcd(\ell,\ell')}-1)\cdot\lcm(\ell,\ell')$.
\end{theorem}

In particular, observe that the only case for which $n$ is not strictly bigger than both $\ell$ and $\ell'$ is the trivial one $b=2$ and $\min(\ell,\ell')=1$, which has no interesting structure since a rational number with a periodic part of length $1$ in base $2$ necessarily belongs to $\D_2$.

The case $b=2$ is also interesting in that it is the only case for which, when $\ell$ and $\ell'$ are mutually prime, the smallest integer $n$ such that $b^n-1$ divides $(b^\ell-1)(b^{\ell'}-1)$ is the same as the one such that $b^n-1$ divides both $b^\ell-1$ and $b^{\ell'}-1$ (see Theorem \ref{gcdProperty999}).

Theorem \ref{99FauxFois99} shows how worse the circular standpoint is compared to usual fractions as regards practical calculation. Observe also that this result, which expresses the (typical) length $n$ of $0.\overline{M}\times 0.\overline{N}$ as a function of the lengths $\ell$ (of $M$) and $\ell'$ (of $N$), is very different from Theorem \ref{gcdProperty999} which gives that the length of the periodic part of $1/(mn)$ is the l.c.m. of the lengths of $1/m$ and $1/n$ (as asserted first by Wallis) when $m$ and $n$ are mutually primes.

The proof of Theorem \ref{99FauxFois99} can be given in a standard way, but we will also provide a more fancy presentation. This latter shows how a convenient notation is sometimes useful to help the understanding of theoretical calculations that are otherwise quite tiresome.

\subsubsection{Notation and beginning of the proof}

We put $\gcd(\ell,\ell')=d$, $\ell=ad$ and $\ell'=a'd$. (Thereofore, $a$ and $a'$ are mutually primes.) Also, we write $\lcm(\ell,\ell')=m$, so $a'\ell=a\ell'=m$.

By Theorem \ref{gcdProperty999}, any $n$ such that $b^n-1$ is a multiple of $b^\ell-1$ belongs to $\ell\N^*$, and the same holds for $\ell'$ instead of $\ell$. Therefore, we have $n\in m\N^*$. What remains to be proved is that the smallest positive integer $k$ for which $b^{km}-1$ is divided by $(b^\ell-1)(b^{\ell'}-1)$ is $k=b^{d}-1$.

First,  divide $b^{km}-1$ by $b^\ell-1$. The formula for the sum of the first terms of a geometric sequence gives 
\[\frac{b^{km}-1}{b^\ell-1}=\frac{b^{ka'\ell}-1}{b^\ell-1}=\sum_{0\leqslant j<ka'}b^{\ell j}=:S_k.\]

Now, our aim is to compute the Euclidean division of $S_k$ by $b^{\ell'}-1$, looking for the smallest $k$ for which the rest is equal to $0$. It is for this part of the proof that a fancy presentation can be helpful. 

\subsubsection{A``fancy'' presentation of the end of the proof}

In numeration in base $b$, the value $S_k$ can be written as the concatenation of $ka'$ copies of the word $0^{\ell-1}1$. Also, in the Euclidean division of $S_k$ by $b^{\ell'}-1$ we are preparing to do, we are not interested in the quotient but in the rest.

Consider one of the digits $1$ in the base $b$-expansion of $S_k$, located at some position $\ell j$ (i.e. corresponding to $b^{\ell j}$), and apply to it the standard way to divide by $b^{\ell'}-1$. The successive steps makes our digit $1$ ``jump'' by $\ell'$ places to the right, again and again, until it attains some place among the $\ell'$ rightmost digits. Now, consider the digit $1$ located at the position $\ell(j-1)$, then apply to it the same algorithm. This new $1$ is now jumping to the right, $\ell'$ places at a time, until it reaches one of the $\ell'$ rightmost places. The (circular) distance between the positions of our two final $1$s is equal to $\ell$ modulo $\ell'$.

For the sake of clarity, assume first that $\ell$ and $\ell'$ are mutually primes. Therefore, the successive $1$ that compose $S_k$ are stacking together in all the possible $\ell'$ rightmost places at the end of their jumps. After the jumps of the first $\ell'$ $1$s, each of these places contains exactly one $1$. After the jump of the next $\ell'$ ones, each of the $\ell'$ rightmost places contains exactly two $1$s, etc. Therefore, for $S_k$ to be divisible by $b^{\ell'}-1$ we need a number of $1$s in $S_k$ multiple of $\ell'\beta$, so that the jumps of all the $1$s eventually produce an expression of the form $(p\beta)^{\ell'}$ at the rightmost places, which is equal to $p\cdot \beta^{\ell'}$. The smallest choice is of course $p=1$, so the right value for $k$ is the value for which the number of $1$s in $S_k$ is exactly $\ell'\beta$. The number of $1$s in $S_k$ being equal to $ka'=k\ell'$ (since $\ell$ and $\ell'$ are mutually primes), we eventually get that $k=\beta$.

Now for the general case of $\gcd(\ell,\ell')=d$. The final positions of the $1$s after their jumps are of the form $id$, for all integers $i$ such that $0\leqslant id<\ell'$. And, as before, the $\ell'/d=a'$ first $1$s displaced are all located at different places. The next $a'$ ones stack on the first $a'$ ones, etc., so after $\beta a'$ $1$s of $S_k$ made their jumps, we get a rest which is the word $(0^{d-1}\beta)^{a'}$. Take $a'$ more $1$s of $S_k$ and we get the word $(0^{d-1}b)^{a'}$, whose proper $b$-expansion writes $(0^{d-2}10)^{a'}$. It is then quite easy to see that we need $\beta^d a'$ digits $1$ in $S_k$ to reach a rest equal to $\beta^{\ell'}$. Since the number of $1$s in $S_k$ is $ka'$, we eventually get that $k=\beta^d=b^{\gcd(\ell,\ell')}-1$.

\subsubsection{Standard writing of the end of the proof}\label{StandardWriting}

Let $j$ be such that $0\leqslant j<ka'$, and let $\ell j=q\ell'+r$ be the Euclidean division of $\ell j$ by $\ell'$, where $0\leqslant r<\ell'$. The Euclidean division of $b^{\ell j}$ by $b^{\ell'}-1$ is therefore
\[b^{\ell j}=\left(b^r\frac{1-b^{q\ell'}}{1-b^{\ell'}}\right)\cdot (b^{\ell'}-1)+b^r.\]

In the sequel, we write $u\ \mbox{mod}\ v$ for the rest of the Euclidean division of $u$ by $v$. Hence we can write

\begin{eqnarray*}
S_k&=&\sum_{0\leqslant j<ka'}b^{\ell j}\\
&=&\sum_{0\leqslant j<ka'}\left(b^{\ell j\ \mbox{\scriptsize mod}\ \ell'}\frac{1-b^{q\ell'}}{1-b^{\ell'}}\right)\cdot (b^{\ell'}-1)+b^{\ell j\ \mbox{\scriptsize mod}\ \ell'}\\
&=&\left(\sum_{0\leqslant j<ka'}b^{\ell j\ \mbox{\scriptsize mod}\ \ell'}\frac{1-b^{q\ell'}}{1-b^{\ell'}}\right)\cdot (b^{\ell'}-1)+\sum_{0\leqslant j<ka'} b^{\ell j\ \mbox{\scriptsize mod}\ \ell'}.
\end{eqnarray*}

Let us write $S'_k$ for the last sum of the latter expression. Of course, $S'_k$ is positive and increases with $k$. Therefore, if we can find a $k$ for which $S'_k=b^{\ell'}-1$, then this will ensure that this $k$ is the one we are looking for.

Recall that $\ell=ad$ and $\ell'=a'd$, so we can write
\begin{eqnarray*}
S'_k&=&\sum_{0\leqslant j<ka'}b^{daj\ \mbox{\scriptsize mod}\ da'}\\
&=&\sum_{0\leqslant j<ka'}\Big(b^{aj\ \mbox{\scriptsize mod}\ a'}\Big)^d\\
&=&\sum_{0\leqslant s<k}\sum_{0\leqslant t <a'}\Big(b^{a(sa'+t)\ \mbox{\scriptsize mod}\ a'}\Big)^d\\
&=&\sum_{0\leqslant s<k}\sum_{0\leqslant t <a'}\Big(b^{at\ \mbox{\scriptsize mod}\ a'}\Big)^d\\
&=&k\sum_{0\leqslant t<a'}\Big(b^{at\ \mbox{\scriptsize mod}\ a'}\Big)^d.
\end{eqnarray*}
Since $a$ and $a'$ are mutually primes, this is equal to ${\displaystyle k \sum_{0\leqslant i<a'}(b^i)^d}$, so 
\[S'_k=k\frac{1-b^{da'}}{1-b^d}=k\frac{1-b^{\ell'}}{1-b^d}.\]
Therefore, the value $k=b^d-1$ makes $S'_k$  reach the value $b^{\ell'}-1$. Hence, $k=b^d-1$ is the value we are looking for, and Theorem \ref{99FauxFois99} is proved.

\subsection{Multiplication of circular words}\label{AlgoMultCirc}

As mentioned in Section \ref{CircMult}, product of circular words of length $\ell$ and $\ell'$ involves the expression $\beta^n/(\beta^\ell\times\beta^{\ell'})$, where $n$ is given by Theorem \ref{99FauxFois99}. The standard writing of the end of the proof of this theorem (Section \ref{StandardWriting}) already provides a part of the answer. Indeed, with the notation in use there, for   $k=b^d-1$ we have 
\[\frac{\beta^n}{\beta^\ell\times\beta^{\ell'}}=1+\sum_{0\leqslant j<(b^d-1)a'}b^{\ell j\ \mbox{\scriptsize mod}\ \ell'}\frac{1-b^{q\ell'}}{1-b^{\ell'}},\]
\noindent the initial $1$ being a consequence of the fact obtained in the end of Section  \ref{StandardWriting} that $S'_k=b^{\ell'}-1$ for $k=b^d-1$.

The words $\widetilde{W}$ and $\widetilde{X}$ being of length $\ell$ and $\ell'$, the circular powers identification allows us to assume $\ell=\ell'$. Now put $B=b^\ell$, so that we may replace the alphabet ${\cal A}=\{0,\ldots, b-1\}$ by ${\cal A}'=\{0,\ldots, b^\ell-1\}$. In practice, such a change of base from $b$ to $B$ consists in grouping digits by blocks of length $\ell$. This will not provide an optimal algorithm, since such a change in notation will lead to a product $\widetilde{W}\times\widetilde{X}$ of length $(b^{aa'd}-1)aa'd$ instead of $(b^d-1)aa'$ (Theorem \ref{99FauxFois99}). Nevertheless, not only we are not really interested in optimality here (since we already know from Proposition \ref{CasPartFauxFois} that it is pointless) but such a simplification will greatly clarify the following. (Without it, we would be led to quite technical considerations about continued fraction expansion of $a/a'$.)

What precedes shows that, up to a change of basis, we may assume that $\ell=\ell'=1$, so $n=b-1$. Our aim is therefore to get an explicit expression of $\beta^n/(\beta^\ell\times\beta^{\ell'})=\beta^{b-1}/(\beta\times\beta)$. The previous equality becomes
\begin{eqnarray*}\frac{\beta^{b-1}}{\beta\times\beta}&=&1+\sum_{0\leqslant j<b-1}\frac{1-b^j}{1-b}\\
&=&1+\sum_{0\leqslant j<b-1}\sum_{0\leqslant i<j}b^i\\
&=&1+\sum_{0\leqslant i<b-2}\sum_{i<j<b-1}b^i\\
&=&1+\sum_{0\leqslant i<b-2}(b-i-2)b^i.
\end{eqnarray*}

Hence, we can eventually define the multiplication of circular words of length 1 (and therefore of any length, by the circular powers identification and a change of base) in the following way:

\begin{definition}\label{ProdCirc} Let $p$ and $p'$ be two letters of ${\cal A}=\{0,\ldots, \beta\}$. The product $\widetilde{p}\times\widetilde{p'}$ is the circular word $\widetilde{Q}$ of length $\beta$  such that
\[N(Q)=pp'\left(1+\sum_{0\leqslant i<b-2}(b-i-2)b^i\right).\]
\end{definition}

Complementary algorithmic considerations could be of some help to shorten the calculation of such a product, since the last factor has a particular form. (For example it is $12345679$ in base ten.) For the same reason as before, we will not consider it here.

\subsection{The fields ${\cal Q}_{\text{\rm WCP}}$ and ${\cal Q}_{\text{\rm DC}}$}\label{AlgoMult}

From Definition \ref{ProdCirc} we can deduce the multiplicative structure on ${\cal Q}_{\text{\rm WCP}}$ and ${\cal Q}_{\text{\rm DC}}$. For the first one, the definition is
\[(W,\widetilde{P},c)\times(W',\widetilde{P'},c'):=(WW'+q(WP')+q(W'P),W\widetilde{P'}+W'\widetilde{P}+\widetilde{P}\times\widetilde{P'},c+c').\]

(According to the assumptions made in Section \ref{AbelianGroupQ}, $q(WP')$ stands for $q_{\ell'}(WP')$ where $\ell'$ is the length of $\widetilde{P'}$. The same remark holds for other expressions of the same kind.)

Now for ${\cal Q}_{\text{\rm DC}}$. For $\delta=kb^c\in\D_b$ (with $k\in\Z$ and $c\in\Z$),  $q(\delta P')$ stands for $q(kP')b^c$, and $\widetilde{r}(\delta P')$ for $\sigma^c(\widetilde{r(kP')}$, where $r(kP')$ is the rest of the Euclidean division of $kP'$ by $\beta^\ell$. With the help of Section \ref{RingD}, we then have
\[(\delta,\widetilde{P})\times(\delta',\widetilde{P'})=(\delta\delta'+q(\delta P')+q(\delta'P), \widetilde{r}(\delta P')+\widetilde{r}(\delta'P)+\widetilde{P}\times\widetilde{P'}).\]

To get the field structure, it remains to show how to define a division. It is not very difficult to generalize the usual algorithm of long division 
to get a combinatorial definition of the division (see (Vivier, 2015) for some examples). Eventually, the fact that we have indeed built the field $\Q$ can be proved by showing that the application $N$ defined by
\[N(s,W,\widetilde{P},c):=s\cdot \left(N(W)+\frac{N(P)}{\beta^\ell}\right)\cdot b^c\]
\noindent is a bijective morphism of rings between ${\cal Q}_{\rm WCP}$ and $\Q$ and
\[N(\delta,\widetilde{P})=\delta+\frac{N(P)}{\beta^\ell}\]
\noindent is another between ${\cal Q}_{\rm DC}$ and $\Q$.

Eventuelly, we get the following

\begin{theorem} ${\cal Q}_{\text{\rm WCP}}$ and ${\cal Q}_{\text{\rm DC}}$ are fields, both isomorphic to $\Q$.

\end{theorem}

\section{Irrationality of algebraic numbers}\label{ThIrr}

As recalled in introduction (see (Bullynck, 2009) for details), Lambert once tried to prove the irrationality of $\pi$ by showing that its decimal expansion is aperiodic. This idea, unsuccessful for $\pi$, can be applied to show the irrationality of simpler numbers, namely the square root of integers. Such a proof seems to be new. Here is a general statement:

\begin{theorem}\label{IrrationaliteUnitaires} Let $Q\in\Z[X]$ be a unitary polynomial. Any (real) root of it is either an integer or an irrational number.

Also, let $Q\in\D_b[X]$ for some integer $b>1$, with coefficient for $X^{\deg(Q)}$ equal to $1$. Any (real) root of it is either in $\D_b$ or irrational.
\end{theorem}

In particular, applying this theorem to the polynomial $Q(X)=X^2-2$ shows that $\sqrt{2}$ is irrational, as well as $\sqrt{n}$ for $n$ not a perfect square (with $Q(X)=X^2-n$), or $\sqrt{2}+\sqrt{3}$ (with $Q(X)=X^4-10X^2+1$).

To apply the second part of the theorem to prove that, for example, $\sqrt[3]{3.57}$ is irrational (with $Q(X)=X^3-3.57$), we have to prove that it is not decimal. For this, we can use the complementary fact that, for any decimal number $\delta$ with exactly $k$ nontrivial digits after the decimal points, $\delta^n$ has exactly $nk$ nontrivial decimal digits after the decimal point. (Indeed, if $k$ is the smallest integer such that $\delta=u/10^k$ for some integer $u$, then, since $10$ is quadratfrei, $\delta^n=u^n/10^{nk}$ cannot be simplified as a fraction with denominator a smaller power of $10$ than $10^{nk}$.) Hence, $\delta:=\sqrt[3]{3.57}$ cannot be in $\D_{\text{\rm ten}}$, otherwise the number $k$ of its digits after the decimal point would satisfy $3k=2$, an impossibility.\footnote{An alternative reasoning would be to consider the last digit of the product.}

In the following, for any rational number $x$, we write $\ell(x)$ for $|\widetilde{P}|$, where $(\delta,\widetilde{P})$ is the DC-representation of $x$ with $\widetilde{P}$ of minimal length.

\begin{proof} We prove first the second part of the theorem. Let $Q\in\D_b[X]$ satisfy the hypotheses, and let $\alpha\in\R$ be a root of $Q$. Assume  $\alpha$  rational, and write $(\delta,\widetilde{P})$ for its DC-expansion in base $b$. It is enough to prove that $\alpha-\delta$ is irrational, so without loss of generality we assume $\delta=0$.

Consider the successive integral powers of $\alpha$. Since $Q$ has dominant coefficient equal to $1$, for any $n\geq q$ there exists values $d_i^{(n)}\in\D_b$ such that ${\displaystyle\alpha^n=\sum_{i=0}^{q-1}d_i^{(n)}\alpha^i}$ (with $q=\deg(Q)$). This implies: 
\begin{equation}
\tag{*} \widetilde{P}^n=\sum_{i=0}^{q-1}d_i^{(n)}\widetilde{P}^i
\end{equation}

(Here, $\widetilde{P}^i$ should not be confused with $\widetilde{P^i}$; by $\widetilde{P}^i$ we mean the iterated product of $\widetilde{P}$ by itself in ${\cal G}^*$.)

To end up the proof, we need the following intermediate result:

\begin{proposition}\label{EllVaALInfini} Let $\widetilde{P}$ be some nontrivial circular word. The sequence $(\ell(\widetilde{P}^n))_n$ goes to infinity.
\end{proposition}

Note that the sequence in Proposition \ref{EllVaALInfini} may not be strictly increasing. Indeed, for $b=7$ and $\widetilde{P}=\widetilde{15}$, we have $\widetilde{P}^2=\widetilde{03}$, so $\ell(\widetilde{P}^2)=\ell(\widetilde{P})=2$.

\begin{proof}[Proof of Proposition \ref{EllVaALInfini}.] The sequence of rational numbers that corresponds to the sequence $((0,\widetilde{P}^n))_n$ is strictly decreasing to $0$. Hence, in the sequence $(\widetilde{P}^n)_n$ it is impossible to find the same circular word twice. Since for any integer $\ell>0$ the number of circular words of length at most $\ell$ is finite, the pigeonhole principle forces the sequence $(\ell(\widetilde{P}^n))_n$ to go to infinity.\footnote{An alternative reasoning consists in showing that the number of $0$s in $\widetilde{P}^n$ goes to infinity.}\end{proof}

Now observe that the length of the right side of $(*)$ is upper-bounded. Hence, by Theorem \ref{EllVaALInfini}, $\widetilde{P}$ is trivial, so $\alpha\in\D_b$, and the second part of the theorem is proved.

For the first part, for $Q\in\Z[X]$ a unitary polynomial, the second part of the theorem applied to all integers $b>1$ shows that any rational root $\alpha$ of $Q$ belongs to ${\displaystyle\alpha\in\bigcap_{b>1}\D_b}=\Z$.\end{proof}

\section{Conclusion}

In this article, based on history, didactics and mathematics, we were interested in a schema (in the {\sc apos} meaning) of the field $\Q$, in the three stages intra-$\Q$, inter-$\Q$ and trans-$\Q$. The equality $0.\overline{9}=1$ seems to be an important proxy of the level of understanding of decimal expansion of rational numbers, since it draws the frontier between the intra- and inter- stages.

In teaching of mathematics, it is common that periodic decimal expansions of rational numbers remains in general at the stage inter-$\Q$, or even intra-$\Q$, the equality $0.\overline{9}=1$ being interpretated in the framework of real numbers (note that we could also define a schema for $\R$). Yet it seems to us that a complete understanding of the schema $\Q$ is an important prerequisite before the construction of $\R$ (even a partial one). For example, constructions by Dedekind cuts or Cauchy sequences can be interpretated as an action on the object $\Q$, obtained by thematization\footnote{Thematization is the mental mecanism that allows to consider a schema as an object, on which one can make actions.} of the schema. The point is that if the schema thus thematized is not complete as regards the periodic decimal expansions, difficulties may arise, since it is crucial to have in mind the equality $0.\overline{9}=1$. Indeed,  $(-\infty,0.\overline{9}]\cup[1,+\infty)$ is a partition of ${\cal Q}_{\text{\rm WCP}}$ which is not suitable as a cut, and the sequence $0.9$, $0.99$, $0.999$, and so on,  is a Cauchy sequence which would not converge in ${\cal Q}_{\text{\rm WCP}}$. (By the way, some introductory notions of topology can also be presented in $\Q$.)

The set $\R$ can also be defined directly from $\D$, considering infinitely many decimals, as in (Lebesgue, 1938) or (Fardin and Li, 2021) (see also Bronner's notion of {\em idecimality} (Bronner, 1997 and 2005)), but this approach has several issues:
\begin{itemize}

\item there is no (finite) algorithm for operations in $\R$ (despite Fardin an Li's attempt), whereas there are in $\Q$, which can provide some sense to the operations made on infinitely many digits.

\item the Totality stage is compulsory for $\R$, but a numeral writing is not always an Process: Totality (in the sense of {\sc apos}) is not, in general, the encapsulation of a process. Sure, some processes are required for writing numbers like $0.101101110111101111\ldots$ or $0.1234567891011\ldots$, and also for writing numbers defined as a limit, like $\mbox{e}$ or $\pi$. This can be eased in $\Q$ by the periodic parts regarded as Objects.

\item $0.\overline{9}=1$ has therefore to be considered in the (rather elaborate) context of real numbers, whereas it could be treated in $\Q$ first.

\end{itemize}

\end{document}